%% file: freesurfacesv3.tex
\documentclass[a4,12pt]{article}
\usepackage{amssymb,amsmath,epsfig,amscd,eucal,psfrag,amsthm,enumerate,mathabx}
\usepackage[normalem]{ulem}
\usepackage[T1]{fontenc}
\usepackage[latin9]{inputenc}
\usepackage{graphicx,color}

\pdfoutput=1

\newtheorem{theorem}{Theorem}[section]

\newtheorem{lemma}[theorem]{Lemma}
\newtheorem{proposition}[theorem]{Proposition}

\newtheorem{corollary}[theorem]{Corollary}

\newtheorem{question}[theorem]{Question}
\newtheorem{letterthm}{Theorem}
\newtheorem{lettercor}[letterthm]{Corollary}

\makeatletter
\newtheorem*{rep@theorem}{\rep@title}
\newcommand{\newreptheorem}[2]{%
\newenvironment{rep#1}[1]{%
 \def\rep@title{#2 \ref{##1}}%
 \begin{rep@theorem}}%
 {\end{rep@theorem}}}
\makeatother

\newreptheorem{theorem}{Theorem}
\newreptheorem{cor}{Corollary}

\theoremstyle{definition}
\newtheorem{definition}[theorem]{Definition}

\theoremstyle{remark}
\newtheorem{remark}[theorem]{Remark}

\newcommand{\Proj}{\mathbb{P}}

\newcommand{\R}{\mathbb{R}}
\newcommand{\ess}{\mathbb{S}}
\newcommand{\Z}{\mathbb{Z}}

\newcommand{\uu}{{\underline{u}}}
\newcommand{\vv}{\underline{v}}
\newcommand{\w}{{\underline{w}}}
\newcommand{\x}{\underline{x}}

\newcommand{\Wh}{\mathrm{Wh}}
\newcommand{\St}{\mathrm{St}}

\newcommand{\wh}{\widehat}
\newcommand{\id}{\mathrm{id}}

\newcommand{\Stab}{\mathrm{Stab}}

\newcommand{\cone}{C_\curlyP}

\newcommand{\curlyA}{\mathcal{A}}
\newcommand{\curlyB}{\mathcal{B}}

\newcommand{\curlyG}{\mathcal{G}}
\newcommand{\curlyH}{\mathcal{H}}

\newcommand{\curlyP}{\mathcal{P}}
\newcommand{\curlyS}{\mathcal{S}}

\newcommand{\curlyU}{\mathcal{U}}
\newcommand{\curlyW}{\mathcal{W}}
\newcommand{\curlyX}{\mathcal{X}}

\newcommand{\into}{\hookrightarrow}
\newcommand{\immerses}{\looparrowright}

\usepackage{hyperref}

\input xy
\xyoption{all}

\title{Essential surfaces in graph pairs}
\author{Henry Wilton\footnote{Supported by EPSRC Standard Grant EP/L026481/1.}}

\newcommand{\Addresses}{{% additional braces for segregating \footnotesize
  \bigskip
  \footnotesize

  \textsc{DPMMS, Centre for Mathematical Sciences, Wilberforce Road, Cambridge, CB3 0WB, UK}\par\nopagebreak
  \textit{E-mail address:} \texttt{h.wilton@maths.cam.ac.uk}

}}

\begin{document}

\maketitle

\begin{abstract}
A well known question of Gromov asks whether every one-ended hyperbolic group $\Gamma$ has a surface subgroup.  We give a positive answer when $\Gamma$ is the fundamental group of a graph of free groups with cyclic edge groups.  As a result, Gromov's question is reduced  (modulo a technical assumption on 2-torsion) to the case when $\Gamma$ is rigid.  We also find surface subgroups in limit groups. It follows that a limit group with the same profinite completion as a free group must in fact be free, which answers a question of Remeslennikov in this case. 
\end{abstract}

This paper addresses a well known question about hyperbolic groups, usually attributed to Gromov.

\begin{question}\label{qu: Gromov}
Does every one-ended hyperbolic group contain a surface subgroup?
\end{question}

Here, a \emph{surface subgroup} is a subgroup isomorphic to the fundamental group of a closed surface of non-positive Euler characteristic.   Various motivations for Gromov's question can be given. It generalizes the famous Surface Subgroup conjecture for hyperbolic 3-manifolds, but it is also a natural challenge when one considers that the Ping-Pong lemma makes free subgroups very easy to construct in hyperbolic groups, whereas a theorem of Gromov--Sela--Delzant \cite{delzant_limage_1995} asserts that a one-ended group has at most finitely many images (up to conjugacy) in a hyperbolic group.  More recently, Markovic proposed finding surface subgroups as a route to proving the Cannon conjecture \cite{markovic_criterion_2013}.

Several important cases of Gromov's question have recently been resolved.  Most famously, Kahn and Markovic proved the Surface Subgroup conjecture \cite{kahn_immersing_2012}.  Extending their work has been the topic of a great deal of recent research (see \cite{hamenstaedt_incompressible_2015,liu_homology_2015}, for instance). In another dramatic development, Calegari and Walker answered Gromov's question affirmatively for random groups \cite{calegari_random_2015}, following similar results for random ascending HNN extensions of free groups (by the same authors \cite{calegari_surface_2015}) and random graphs of free groups with edge groups of rank at least two (by Calegari and the author \cite{calegari_random_2013a}).

In this paper, we resolve Gromov's question for a contrasting class of hyperbolic groups -- graphs of free groups with cyclic edge groups.  Our main theorem answers Gromov's question affirmatively in this case.

\begin{letterthm}\label{thm: GOFGWCEG}
Let $\Gamma$ be the fundamental group of a graph of free groups with cyclic edge groups.  If $\Gamma$ is one-ended and hyperbolic then $\Gamma$ contains a quasiconvex surface subgroup.
\end{letterthm}

In fact, using a result of Wise, we are able to find surface subgroups in graphs of \emph{virtually} free groups with \emph{virtually} cyclic edge groups; see Theorem \ref{thm: Surface subgroups} below.

Numerous special cases of this result are already known.  Calegari used his work on the rationality of stable commutator length in free groups \cite{calegari_stable_2009} to show that surface subgroups exist when $H_2(\Gamma;\mathbb{Q})\neq 0$ \cite{calegari_surface_2008}.  Infinite classes of examples were found by the author in joint works with Gordon \cite{gordon_surface_2010} and with Kim \cite{kim_polygonal_2012}.  Kim and Oum found surface subgroups in doubles of free groups of rank two \cite{kim_hyperbolic_2014}.  The author answered a weaker version of Gromov's question for this class of groups, by showing that every such $\Gamma$ is either a surface group or contains a finitely generated, one-ended subgroup of infinite index \cite{wilton_one-ended_2011}.    
  
Although the class of hyperbolic groups covered by Theorem \ref{thm: GOFGWCEG} is quite specific, the theorem has wider consequences for Gromov's question.  We call a group \emph{rigid} if it does not admit a non-trivial splitting with a virtually cyclic edge group.    Using strong accessibility \cite{louder_strong_2017} we can, modulo a technical hypothesis on 2-torsion, reduce Gromov's question to  the rigid case, using the following corollary.

\begin{lettercor}\label{lettercor: Surface or rigid}
Let $\Gamma$ be a one-ended hyperbolic group without 2-torsion.  Either $\Gamma$ contains a quasiconvex surface subgroup, or $\Gamma$ contains a quasiconvex rigid subgroup.
\end{lettercor}

See Corollary \ref{cor: Surface or rigid} for full details.  By a theorem of Bowditch \cite{bowditch_cut_1998}, a one-ended hyperbolic group $\Gamma$ is rigid if and only if its Gromov boundary does not contain local cut points (unless $\Gamma$ is a finite extension of a triangle group).  Corollary \ref{lettercor: Surface or rigid} should be useful in any attempt at a general answer to Gromov's question, since local cut points in the boundary present extra technical challenges for the ergodic techniques of \cite{kahn_immersing_2012} and the probabilistic techniques of \cite{calegari_random_2015}, as witnessed by the difficulties resolved in Kahn and Markovic's proof of the Ehrenpreis conjecture \cite{kahn_good_2015}.

A \emph{limit group} is a finitely generated, fully residually free group -- that is, a finitely generated group in which every finite subset can be mapped injectively into a free group by a group homomorphism.  Limit groups play a central role in the study of algebraic geometry and logic over free groups; see \cite{sela_diophantine_2001} \emph{et seq.},\ in which they were defined, and also the parallel project \cite{kharlampovich_tarskis_1998} \emph{et seq.}  Theorem \ref{thm: GOFGWCEG} addresses the key case for the problem of finding surface subgroups of limit groups, and so we can also answer Gromov's question in that context.

\begin{lettercor}\label{lettercor: Limit groups}
Let $\Gamma$ be a limit group.  If $\Gamma$ is one-ended then $\Gamma$ contains a surface subgroup.
\end{lettercor}

See Corollary \ref{cor: Surface subgroups of limit groups} for details. Note that limit groups are not all hyperbolic, but they are all toral relatively hyperbolic \cite{alibegovic_combination_2005,dahmani_combination_2003}.  In particular, non-hyperbolic limit groups contain a $\Z^2$ subgroup, so the hyperbolic case is the one of interest. 

These results have interesting applications to a different structural problem in group theory.   Recall that the \emph{profinite completion}, $\wh{\Gamma}$, of a group $\Gamma$ is the closure of the image of $\Gamma$ in the direct product of its finite quotients (endowed with the product topology).   If two groups $\Gamma_1$ and $\Gamma_2$ have isomorphic profinite completions, then it is natural to ask whether $\Gamma_1$ and $\Gamma_2$ must be isomorphic. 

In general, the anwer is `no'.  There are even examples of non-isomorphic pairs of virtually cyclic groups with isomorphic profinite completions \cite{baumslag_residually_1974}.  Nevertheless, many important questions of this type remain open, of which the following question of Remeslennikov is one of the most notable \cite[Question 15]{noskov_infinite_1979}.

\begin{question}[Remeslennikov]\label{qu: Remeslennikov}
Suppose that $F$ is a finitely generated, non-abelian free group and that $\Gamma$ is finitely generated and residually finite. If $\wh{\Gamma}\cong\wh{F}$, does it follow that $\Gamma\cong F$?
\end{question}

It is particularly natural to consider Question \ref{qu: Remeslennikov} when $\Gamma$ is a limit group. Indeed, limit groups are closely related to free groups (for instance, Remeslennikov showed that they are precisely the \emph{existentially free groups} \cite{remeslennikov_exists-free_1989}), and are frequently hard to distinguish from them.  Bridson, Conder and Reid \cite{bridson_determining_2016} pointed out that Corollary \ref{lettercor: Limit groups}, combined with the results of \cite{wilton_halls_2008}, would resolve Remeslennikov's question in this case.

\begin{lettercor}\label{introcor: Remeslennikov for limit groups}
If $L$ is a limit group and not free then the profinite completion $\wh{L}$ is not isomorphic to the profinite completion of any free group.
\end{lettercor}

The same ideas give a new proof of a theorem of Puder and Parzanchevski \cite[Corollary 1.5]{puder_measure_2015}. Recall that a word $w$ in a free group $F$ is called \emph{primitive} if $F$ splits as a free product $\langle w\rangle*F_1$.  Similarly, an element $\hat{w}$ of the profinite free group $\wh{F}$ is called \emph{primitive} if $\wh{F}$ decomposes as a  coproduct $\overline{\langle \hat{w}\rangle}\coprod\wh{F}_1$ in the category of profinite groups.  Puder and Parzanchevski showed that an element $w$ of $F$ that is primitive in the profinite completion $\wh{F}$ is already primitive in $F$ \cite[Corollary 1.5]{puder_measure_2015}.  This can be thought of as  answering a relative version of Question \ref{qu: Remeslennikov}.

In fact, we can generalize their result, from words to multiwords (i.e.\ finite indexed sets of words).  Let us call a multiword $\w=\{w_1,\ldots,w_n\}$ in $F$ \emph{primitive} if $F$ splits as a free product $\langle w_1\rangle*\ldots*\langle w_n\rangle*F_1$ for some $F_1$, and make the corresponding definition of a primitive multiword in $\wh{F}$.  

\begin{lettercor}\label{introcor: PP}
Let $F$ be a finitely generated free group.  If a multiword $\w$ is primitive in the profinite completion $\wh{F}$ then it is primitive in $F$.
\end{lettercor}

It is interesting to contrast the techniques of this paper with those of \cite{puder_measure_2015}.  Puder and Parzanchevski deduce their result from their beautiful characterization of primitive words in free groups as precisely the \emph{measure-preserving words} \cite[Theorem 1.1]{puder_measure_2015}.  The proof given here is cohomological, and goes via the fact that the virtual second cohomology of the profinite completion of a non-free limit group is non-zero.  Corollaries \ref{introcor: Remeslennikov for limit groups} and \ref{introcor: PP} follow quickly from Theorem \ref{thm: Profinite vb2 for limit groups}. We refer the reader to that theorem and the subsequent remarks for details.

Let us now turn to discuss the proof of Theorem \ref{thm: GOFGWCEG}.    Our main technical result addresses a \emph{relative} version of Gromov's question, finding surfaces in free groups relative to families of cyclic subgroups.  To state it concisely we need to introduce some definitions.

Consider a graph of spaces $X$ in the sense of Scott and Wall \cite{scott_topological_1979}, and let $v$ be a vertex with incident edges $e_1,\ldots,e_n$. The vertex space $X_v$, together with the maps of incident edge spaces $w_i:X_{e_i}\to X_v$, defines a \emph{space pair} $(X_v,\w)$.   In the case of interest, the vertex space $X_v$ will always be a graph (usually denoted by $\Gamma$), and the edge spaces $X_{e_i}$ will be circles $S^1_i$; such a $(\Gamma,\w)$ is called a \emph{graph pair}.

Global properties of the graph of spaces $X$ can be characterized locally, using  properties of the pairs associated to vertex spaces.  For instance, $X$ is called \emph{irreducible} if $\pi_1X$ does not split over a finite subgroup; we may correspondingly define an \emph{irreducible pair} $(\Gamma,\w)$ (see Definition \ref{defn: Irreducible pair}), and a lemma of Shenitzer asserts that if the pairs associated to the vertices are irreducible, then so is $X$ \cite[Theorem 18]{wilton_one-ended_2011}.  Corresponding to the notion of a $\pi_1$-injective map of graphs of spaces $Y\to X$, we have an \emph{essential} map of pairs $(\Lambda,\uu)\to (\Gamma,\w)$, and indeed if a morphism of graphs of spaces is essential on each space pair associated to a vertex then the morphism is itself $\pi_1$-injective (see Proposition \ref{prop: Essential maps of graphs of groups}).

We are now ready to state the main technical result.

\begin{letterthm}\label{thm: Surfaces in graphs}
If $(\Gamma,\w)$ is an irreducible graph pair then there is a compact surface with boundary $\Sigma$ and an essential map of pairs $(\Sigma,\partial\Sigma)\to(\Gamma,\w)$.
\end{letterthm}

In fact, we obtain a bit more control than this -- the surface is also \emph{admissible}, meaning that every point of the domain of  $\w$ has the same number of preimages in $\partial\Sigma$;  see Theorem \ref{thm: Essential surfaces exist}.  It has been well known for a while that a result like Theorem \ref{thm: Surfaces in graphs} would imply the existence of surface subgroups in graphs of free groups with cyclic edge groups -- see, for instance, \cite{calegari_surface_2008} or \cite{kim_polygonal_2012}.

Let us now briefly sketch the proof of Theorem \ref{thm: Surfaces in graphs}.  It can be thought of as a combination of the techniques of \cite{calegari_stable_2009} and \cite{wilton_one-ended_2011}.

First, we study irreducible pairs $(\Lambda,\uu)$ that map into the irreducible pair $(\Gamma,\w)$.     The irreducibility of the pairs $(\Gamma,\w)$ and $(\Lambda,\uu)$ is characterized using \emph{Whitehead graphs}.  We would like to study essential maps $(\Lambda,\uu)\to (\Gamma,\w)$, but it turns out to be difficult to simultaneously characterize both the fact that $(\Lambda,\uu)$ is irreducible and the fact that the map $(\Lambda,\uu)\to(\Gamma,\w)$ is essential.   In order to recognize both these properties simultaneously we work with \emph{$\partial$-immersions}, which are compositions $(\Lambda,\uu)\to(\Delta,\vv)\to (\Gamma,\w)$.  We can recognize if the pair $(\Lambda,\uu)$ is \emph{locally} irreducible, and this guarantees that $(\Delta,\vv)$ is (weakly) irreducible.

The idea behind the proof of Theorem \ref{thm: Surfaces in graphs} is that, among all irreducible pairs mapping to $(\Gamma,\w)$, the pairs of surface type should be the ones of most negative Euler characteristic.  To make this precise, we define a positive polyhedral cone $\cone$ in a finite-dimensional vector space, such that the integer points in $\cone$ correspond to admissible $\partial$-immersions of (weakly) irreducible pairs $(\Delta,\vv)$.     We also define the \emph{projective $\curlyP$-rank function} $\rho_\curlyP$ on the projectivization $\Proj(\cone)$ as a quotient of two linear functionals: the (negation of the) Euler characteristic of $\Delta$, and the degree with which $\vv$ covers $\w$.   In particular, $\rho_\curlyP$ achieves its maximum value at some vertex of the polyhedron  $\Proj(\cone)$, which is necessarily a rational line in $\cone$.  Since this rational line contains an integer point, an admissible $\partial$-immersion of an irreducible pair exists that maximizes $\rho_\curlyP$.  We call such a pair  \emph{maximal}.

 This approach is similar to the argument of \cite{calegari_stable_2009}, in which a polyhedral cone is defined whose integer points correspond to certain maps of surfaces $(\Sigma,\partial\Sigma)\to(\Gamma,\w)$.    The hypothesis in \cite{calegari_surface_2008} that rational second homology is non-zero is needed to ensure that this cone is non-zero.  In contrast, the cone $\cone$ is guaranteed to be non-zero since, whenever $(\Gamma,\w)$ is irreducible, the identity map $(\Gamma,\w)\to(\Gamma,\w)$ leads to an admissible $\partial$-immersion of an irreducible pair.

The final step of the proof applies the ideas of \cite{wilton_one-ended_2011} to the relative JSJ decomposition of a maximal pair $(\Delta,\vv)$. The conclusion is that any maximal pair has no rigid vertices in its JSJ decomposition.  It follows that the JSJ decomposition is built from surface pieces, and one quickly concludes that a pair of surface type exists.  Thus, we deduce the existence of an admissible, $\partial$-essential surface $(\Sigma,\partial\Sigma)\to(\Gamma,\w)$. 

The paper is structured as follows. In Section \ref{sec: Pairs}, we define pairs of groups, spaces and graphs, the natural notions of maps between them, and various properties of those maps.  In Section \ref{sec: Whitehead graphs} we adapt the classical theory of Whitehead graphs to the setting of a graph pair  $(\Gamma,\w)$.  The main result here is a converse to Whitehead's lemma (Lemma \ref{lem: Converse to Whitehead}), which asserts that an irreducible pair can always be unfolded to a \emph{locally} irreducible pair, in which the irreducibility is recognized by the Whitehead graphs at the vertices.  In Section \ref{sec: Admissible pairs}, we characterize admissible $\partial$-immersions from locally irreducible graph pairs into $(\Gamma,\w)$ as precisely those maps that can be built from a certain finite set $\curlyP$ of \emph{pieces}.  We define the cone $\cone$ and note that there is a surjective map from admissible $\partial$-immersions of locally irreducible pairs to the integer points of $\cone$.  In Section \ref{sec: Rationality}, we define the \emph{projective $\curlyP$-rank} function $\rho_\curlyP$, and prove that it attains its extremal values at rational points of $\Proj(\cone)$.  We deduce the existence of a maximal, admissible $\partial$-immersion from a locally irreducible pair.  In Section \ref{sec: Surfaces}, we apply the results of \cite{wilton_one-ended_2011} to study admissible $\partial$-immersions with maximal projective $\curlyP$-rank.  The main result is that there is such a maximal pair of (weak) surface type (Theorem \ref{thm: Maximal pairs and surfaces}).  Theorem \ref{thm: Essential surfaces exist}, and hence Theorem \ref{thm: Surfaces in graphs}, follow quickly.  In Section \ref{sec: Hierarchies} we deduce Theorem \ref{thm: GOFGWCEG} and Corollaries \ref{lettercor: Surface or rigid} and \ref{lettercor: Limit groups}.   Finally, in Section \ref{sec: Profinite rigidity}, we deduce Corollaries \ref{introcor: Remeslennikov for limit groups} and \ref{introcor: PP}.

\subsection*{Acknowledgements}

Alan Reid asked me whether one-ended limit groups have surface subgroups in 2006.  I have worked on finding surface subgroups with a variety of collaborators, and am grateful to them all: Lars Louder, Sang-hyun Kim, Cameron Gordon, Danny Calegari, Ben Barrett.  Not all of these projects led to publications, but I learned a lot from each of them.  A conversation with Fr\'ed\'eric Haglund and Pierre Pansu led to the discovery of a serious mistake in an earlier attempted proof.  Thanks are also due to Daniel Groves for comments on an early version of this paper. I am especially grateful to Lars Louder for spotting a subtle error in the first version of this paper.

\section{Pairs}\label{sec: Pairs}

We will make heavy use of graphs of groups and Bass--Serre theory, as detailed in Serre's standard work on the subject \cite{serre_arbres_1977}, to which the reader is referred for details.  To fix notation, we recall the definition of a graph.

\begin{definition}
A \emph{graph} $\Gamma$ consists of a vertex set $V$, an edge set $E$, a fixed-point free involution $E\to E$ denoted by $e\mapsto\bar{e}$, and an \emph{origin map} $\iota:E\to V$.  The \emph{terminus map} $\tau:E\to V$ is defined by $\tau(e)=\iota(\bar{e})$.  
\end{definition}

The edges of $\Gamma$ are thus equipped with orientations, and the unoriented edges are the pairs $\{e,\bar{e}\}$.

As well as using graphs of groups, we will also frequently adopt the topological point of view, in which a graph of groups is viewed as the fundamental group of a graph of spaces \cite{scott_topological_1979}.  Graphs of spaces are not required to be connected, which will present some technical advantages, although the attaching maps are required to be injective on fundamental groups.  Analogously, we may also work with disconnected graphs of groups, as long as we are careful to choose a base point before talking about the fundamental group.

\subsection{Group pairs}

It is particularly important for us to work with \emph{relative} versions of graphs of groups and spaces, which characterize the relationship between a vertex group (or space) and its incident edge groups (or spaces).  To this end, we define various notions of pairs.  We start with pairs of groups.

\begin{definition}
A \emph{group pair} is a pair $(G,\curlyA)$, where $G$ is a group and $\curlyA$ is a $G$-set.  It is often convenient to choose a finite set of orbit representatives $\{a_i\}$, to let $H_i=\Stab_G(a_i)$, and to specify the pair via the data $(G,\{H_i\})$.   We will use both the notations $(G,\curlyA)$ and  $(G,\{H_i\})$ to specify group pairs, without fear of confusion.
\end{definition}

The key example of a group pair arises when considering a vertex $v$ of a graph of groups $\curlyG$.  Having fixed a lift $\tilde{v}$ of $v$ to the Bass--Serre tree, one takes $
G$ to be the vertex stabilizer $\curlyG_{\tilde{v}}$ and $\curlyA$ to be the set of edges incident at $\tilde{v}$.     

\begin{definition}
A \emph{morphism of graphs of groups} is a morphism of the underlying graphs, accompanied by associated maps of vertex groups and edge groups that intertwine with the attaching maps.    This is most easily thought of by passing to the Bass--Serre tree.  A morphism of graphs of groups  induces a homomorphism of fundamental groups, and lifts to an equivariant map on Bass--Serre trees.
\end{definition}

This motivates the following definition for pairs.

\begin{definition}
A \emph{morphism of group pairs} $(f,\phi):(G,\curlyA)\to (G',\curlyA')$ consists of a set map $\phi:\curlyA\to\curlyA'$ and a homomorphism $f:G\to G'$ that intertwines $\phi$.  That is, we require that
\[
\phi(g.a)=f(g).\phi(a)
\]
for all $g\in G$ and $a\in\curlyA$.
\end{definition}

In particular, a morphism of graphs of groups defines morphisms of the group pairs at each vertex, and conversely morphisms of pairs that satisfy an obvious compatibility condition can be pieced together to give a morphism of a graph of groups.

It is convenient if we can detect \emph{global} properties of morphisms of graphs of groups by looking at \emph{local} properties of the induced maps on group pairs. We are particularly concerned with $\pi_1$-injectivity, and so we need to develop corresponding notions for group pairs.    Requiring that the map of groups $f:G\to G'$ be injective is clearly significant. The following condition is also important.

%Commented out the old definition, which I don't think is quite right.

\iffalse
\begin{definition}
A morphism of group pairs $(G,\curlyA)\to (G',\curlyA')$ is \emph{$\partial$-essential} if it satisfies the following two conditions.
\begin{enumerate}[(i)]
\item The induced map on quotients $G\backslash\curlyA\to G'\backslash\curlyA'$ is injective.
\item For all $a\in\curlyA$, 
\[
\mathrm{Stab}_{f(G)}(\phi(a))=f(\mathrm{Stab}_G(a))~.
\]
\end{enumerate}
\end{definition}
\fi

\begin{definition}
A morphism of group pairs $(f,\phi):(G,\curlyA)\to (G',\curlyA')$ is \emph{$\partial$-essential} if the map
\[
\ker f\backslash\curlyA\to\curlyA'
\]
induced by $\phi$ is injective.
\end{definition}

When applied to pairs associated to graphs of groups, this condition guarantees that the induced map on Bass--Serre trees does not factor through a fold. Putting this together with injectivity, we have the notion of an essential morphism.

\begin{definition}
A morphism of group pairs $(f,\phi):(G,\curlyA)\to (G',\curlyA')$ is \emph{essential} if the homomorphism $f:G\to G'$ is injective and the morphism is also $\partial$-essential.
\end{definition}

From this one easily deduces a local criterion for morphisms of graphs of groups to be $\pi_1$-injective.

\begin{proposition}\label{prop: Essential maps of graphs of groups}
Suppose that $f$ is a morphism of graphs of groups.  If $f$ induces essential morphisms on the group pairs corresponding to vertices, then $f$ induces an injective map on fundamental groups.
\end{proposition}
\begin{proof}
Suppose that a group element $g$ is in the kernel of $f$. Since the map on Bass--Serre trees does not factor through a fold, if $g$ acts hyperbolically on the Bass--Serre tree then so does its image, contradicting the fact that $g$ is in the kernel.  Therefore $g$ is elliptic, but since $f$ is injective on vertex stabilizers, it follows that $g=1$.
\end{proof}

\subsection{Space pairs}

We next make analogous definitions for spaces.

\begin{definition}
A \emph{space pair} consists of cell complexes $X,Y$ together with a continuous map $\w:Y\to X$.  We will frequently take $\pi_0 Y$ to be an index set $I$, and let
\[
w_i:Y_i\to X
\]
denote the restriction of $\w$ to $Y_i$, the path component of $Y$ corresponding to $i\in I$.   We will often use the notation $(X,\w)$ to denote such a space pair.
\end{definition}

In most of what follows, we will take $X$ to be a graph and $Y$ to be a disjoint union of circles. However, it is useful to allow the extra flexibility of the general definition. 

If $X$ is path connected then a space pair $(X,\w)$ naturally defines a group pair $(G,\curlyA)$.   Let $p:\widetilde{X}\to X$ be the universal cover, and consider the fibre product
\[
\widetilde{Y}:=\widetilde{X}\times_X Y=\{(\tilde{x},y)\mid p(\tilde{x})=\w(y)\}~.
\]
Taking $G=\pi_1X$ and $\curlyA=\pi_0\widetilde{Y}$, we see that $G$ acts naturally on $\curlyA$, and so $(G,\curlyA)$ is a group pair.

This definition is more transparent if one thinks of $X$ as a vertex space of a graph of spaces $Z$.  The universal cover $\widetilde{Z}$ of $Z$ inherits a decomposition as a graph of spaces; $\widetilde{X}$ appears as a vertex space of $\widetilde{Z}$, and the fibre product $\widetilde{Y}$ is the disjoint union of the edge spaces of $\widetilde{Z}$ incident at $\widetilde{X}$.

We next define morphisms of space pairs, analogously to morphisms of group pairs.

\begin{definition}
Let $\w:Y\to X$ and $\w':Y'\to X'$ define space pairs $(X,\w)$ and $(X,\w')$.  A \emph{morphism of space pairs} $(X,\w)\to(X',\w')$ consists of continuous maps $\phi:Y\to Y'$ and $f:X\to X'$ so that $f\circ\w=\w'\circ \phi$.
\end{definition}

As in the case of groups, compatible collections of maps of pairs can be glued together to construct a map of graphs of spaces.  Again, we will need a definition of a $\partial$-essential morphism.

\begin{definition}
Consider a morphism of space pairs $(X,\w)\to(X',\w')$.  Let $\widetilde{X}'$ be the universal cover of $X'$ and let $\widehat{X}$ be the corresponding covering space of $X$, obtained by pulling back the covering map $\widetilde{X}'\to X'$ along $f$.  Consider the fibre products
\[
\widehat{Y}= \widehat{X}\times_X Y~,~\widetilde{Y}'= \widetilde{X}'\times_{X'} Y'~,
\]
and note that the map $\phi:Y\to Y'$ lifts to a map $\tilde{\phi}:\widehat{Y}\to \widetilde{Y}'$.  The morphism $(X,\w)\to(X',\w')$ is called \emph{$\partial$-essential} if $\tilde{\phi}$ induces an injective map $\pi_0\widehat{Y}\to\pi_0\widetilde{Y}'$.
\end{definition}

Again, we combine this with injectivity on vertex groups to obtain a notion of an essential morphism.

\begin{definition}
A morphism of space pairs  $f:(X,\w)\to(X',\w')$ is \emph{essential} if it is $\partial$-essential and $f:X\to X'$ is $\pi_1$-injective.
\end{definition}

Finally, we note that our two definitions of $\partial$-essential pairs coincide.

\begin{lemma}
Let $(X,\w)\to (X',\w')$ be a morphism of space pairs, inducing the corresponding morphism of pairs $(G,\curlyA)\to (G',\curlyA')$ on fundamental groups.   The morphism $(X,\w)\to (X',\w)$ is $\partial$-essential if and only if the morphism $(G,\curlyA)\to (G',\curlyA')$ is $\partial$-essential.  Hence, $(X,\w)\to (X',\w)$ is essential if and only if $(G,\curlyA)\to (G',\curlyA')$ is essential.
\end{lemma}
\begin{proof}
The quotient of $\widetilde{X}$ by the action of $\ker f$ is $\widehat{X}$, and the corresponding covering map induces a map $\widetilde{Y}\to\widehat{Y}$.  

We need to show that two path components $\widetilde{Y}_1$ and $\widetilde{Y}_2$ of $\widetilde{Y}$ have the same image under this map if and only if they are in the same orbit of $\ker f$.  The `if' direction is clear. For the converse, we choose compatible basepoints $*_i\in\widetilde{Y}_i$ and suppose that $\widetilde{Y}_1$ and $\widetilde{Y}_2$ have the same image. Then the images of their basepoints in $\widetilde{X}$ are joined by a concatenation of paths $\kappa\cdot\eta$ where $\kappa$ maps to a loop in $\widehat{X}$ and $\eta$ is the image of a lift of a loop from $Y$ to $\widetilde{Y}_2$.  These define group elements $k\in\ker f$ and $y\in \mathrm{Stab}_G(\widetilde{Y}_2)$ so that $gy$ translates $\widetilde{Y}_1$ to $\widetilde{Y}_2$, and so $\widetilde{Y}_1$ and $\widetilde{Y}_2$ are indeed in the same orbit of $\ker f$.

The lemma follows immediately.
\end{proof}

\subsection{Graph pairs}

In the setting of Theorem \ref{thm: GOFGWCEG}, the groups $G$ are finitely generated free groups, so the spaces $X$ can be taken to be graphs.  We may therefore apply the techniques of Stallings \cite{stallings_topology_1983}.   

\begin{definition}
Let $\Gamma$ be a graph.  The \emph{star} of a vertex $v$ is the set $\St_\Gamma(v)=\{e\in E\mid \iota(e)=v\}$, the set of edges with initial vertex $v$.  (We will also write $\St(v)$ for $\St_\Gamma(v)$ when there is no fear of confusion.)  A morphism of graphs $f:\Gamma\to\Delta$ is an \emph{immersion} if the induced maps on stars are injective.  In this case, we write $f:\Gamma\looparrowright\Delta$.
\end{definition}

Stallings famously observed that immersions are $\pi_1$-injective, and that any morphism of finite graphs $\Lambda\to\Gamma$ factors through a canonical immersion
\[
\Lambda\to \Lambda_0\looparrowright\Gamma
\]
where the map $\Lambda\to\Lambda_0$ is a composition of finitely many folds \cite[\S\S 3.3]{stallings_topology_1983}.

\begin{definition}
A \emph{multicycle} in a graph $\Gamma$ is an immersion of graphs $\w:\ess_{\w} \to \Gamma$, where $\ess_{\w}$ is a disjoint union of graphs homeomorphic to circles.  The components of $\ess_\w$ are denoted by $S^1_i$ and the restriction of $\w$ to $S^1_i$ is denoted by $w_i$.

A \emph{graph pair} is a space pair $(\Gamma,\w)$, where $\Gamma$ is a finite graph without vertices of valence one and $\w$ is a multicycle.  Note that we do not require the graph $\Gamma$ to be connected.
\end{definition}

Again, we will need a notion of morphism for graph pairs.  As for Stallings, for us a morphism of graphs takes vertices to vertices and edges to edges.  Since we insist that the maps $\w$ are immersions, we make a corresponding requirement for morphisms of graph pairs.

\begin{definition}
Let  $(\Gamma,\w)$ and $(\Lambda,\uu)$ be graph pairs.  A morphism of space pairs $(\Lambda,\uu)\to (\Gamma,\w)$ is a \emph{morphism of graph pairs} if the map $\Lambda\to \Gamma$ is a morphism of graphs and the map $\ess_\uu\to\ess_\w$ is an immersion.
\end{definition}

The first advantage of this setting is that we can certify $\pi_1$-injective maps using immersions.   Note that a morphism of graphs $\Lambda\to\Gamma$ is an immersion if and only if the lift to universal covers $\widetilde{\Lambda}\to\widetilde{\Gamma}$ is injective.   Similarly, we may define an immersion of graph pairs.

\begin{definition}
A map of graphs pairs $f:(\Lambda,\uu)\to(\Gamma,\w)$ is an \emph{immersion} if the lifts $\widetilde{\Lambda}\to\widetilde{\Gamma}$ and $\widetilde{\ess}_\uu\to\widetilde{\ess}_\w$ are injective.  In this case, we write $f:(\Lambda,\uu)\looparrowright(\Gamma,\w)$.
\end{definition}

A map of graph pairs factors through a canonical immersion, just as maps of graphs do.

\begin{lemma}\label{lem: Folded pair}
A map of graphs pairs $f:(\Lambda,\uu)\to(\Gamma,\w)$ factors through a canonical immersion $f_0:(\Lambda_0,\uu_0)\looparrowright (\Gamma,\w)$.  The immersion $f_0$ has the universal property that, whenever $f$ factors through an immersion $(\Delta,\vv)\looparrowright(\Gamma,\w)$, $f_0$ also factors through $(\Delta,\vv)\looparrowright(\Gamma,\w)$.
\end{lemma}
\begin{proof}
Let $G$ be the image of $\pi_1\Lambda$ in $\pi_1\Gamma$.  We take $\widetilde{\Lambda}_0$ to be the image of $\widetilde{\Lambda}$ in $\widetilde{\Gamma}$; likewise, we take  $\widetilde{\ess}_{\uu_0}$ to be the image of $\widetilde{\ess}_\uu$ in $\widetilde{\ess}_\w$. The group $G$ acts naturally on each of these, and we take $\Lambda_0$ and $\ess_{\uu_0}$ to be the respective quotients by the action of $G$.
\end{proof}

The next lemma provides a means of locally certifying that a map is $\partial$-essential.

\begin{lemma}\label{lem: d-essential}
If a map of graph pairs $f:(\Lambda,\uu)\to (\Gamma,\w)$ is $\partial$-essential and $(\Lambda_0,\uu_0)\immerses (\Gamma,\w)$ is the corresponding canonical immersion then the induced map $\ess_\uu\to\ess_{\uu_0}$ is injective. Conversely, if a map of graph pairs $f:(\Lambda,\uu)\to (\Gamma,\w)$ factors through an immersion as
\[
(\Lambda,\uu)\to (\Delta,\vv)\looparrowright(\Gamma,\w)
\]
and $\ess_\uu\to\ess_{\vv}$ is injective then $f$ is $\partial$-essential.
\end{lemma}
\begin{proof}
By definition, if $f$ is $\partial$-essential then the corresponding map $\widehat{\ess}_{\uu}\to \widetilde{\ess}_\w$ is injective on $\pi_0$. The components of these spaces are lines, so the map $\widehat{\ess}_{\uu}\to \widetilde{\ess}_\w$ is itself injective, and so $\widehat{\ess}_{\uu}\to \widetilde{\ess}_{\uu_0}$ is injective too.  Finally, since $\ess_{\uu}\to\ess_{\uu_0}$ is obtained by quotienting the domain and the range by $f_*\pi_1\Lambda$, it is also injective.

For the converse, if $f$ factors as hypothesized, then the lift of $\ess_\uu\to\ess_\w$ factors as
\[
\widehat{\ess}_\uu\to\widetilde{\ess}_{\vv}\to\widetilde{\ess}_\w~.
\]
The first map is a lift of an injection, hence an injection, and the second map is a lift of an immersion, hence injective.  The result follows since a composition of injective maps is injective.
\end{proof}

Thus, we can use a map to an immersed pair as a certificate that a morphism of pairs is $\partial$-essential.  We call the data of this certificate a \emph{$\partial$-immersion}.

\begin{definition}\label{defn: d-immersion}
A \emph{$\partial$-immersion} is a concatenation
\[
(\Lambda,\uu)\to (\Delta,\vv)\immerses(\Gamma,\w)
\]
where $\ess_\uu\to\ess_{\vv}$ is bijective.
\end{definition}

\section{Whitehead graphs and folds}\label{sec: Whitehead graphs}

Given a group pair $(F,\langle w\rangle)$, where $w$ is some non-trivial element of a free group $F$, it is natural to ask ask whether or not $\langle w\rangle$ is a free factor of $F$. The standard way of answering this question uses the Whitehead graph, which was defined by J. H. C.  Whitehead in his original paper on automorphisms of free groups  \cite{whitehead_equivalent_1936}.  (See also \cite{cashen_line_2011} and the references therein for a modern account of Whitehead graphs.)  The definition of Whitehead graph given in \cite{whitehead_equivalent_1936}  implicitly involves representing $F$ as the fundamental group of a rose -- a graph with a single vertex.  

Here, we develop the theory of Whitehead graphs for general graph pairs.  We are not aware that this approach has been taken in the literature before, but it is similar to the approaches to Whitehead graphs given by Cashen--Macura \cite{cashen_line_2011} and Manning \cite{manning_virtually_2010}.

\begin{figure}[h]
\begin{center}
 \centering \def\svgwidth{300pt}
 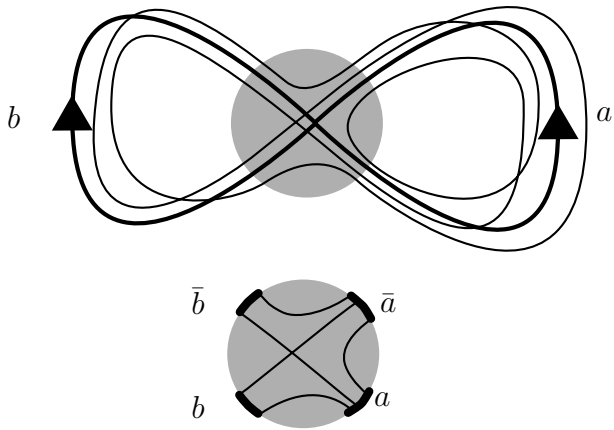
 \caption{The Whitehead graph of the Baumslag--Solitar word $b^{-1}aba^{-2}$.}
\end{center}
\end{figure}

\begin{definition}
Consider a graph pair $(\Gamma,\w)$ and a vertex $x$ of $\Gamma$. The \emph{Whitehead graph} at $x$ is denoted by $\Wh_x(\w)$.  Its set of vertices is the star $\St_\Gamma(x)$.  The unoriented edges of $\Wh_x(\w)$ are the vertices $\{x_i\}$ of $\ess_\w$ that map to $x$; the edge corresponding to $x_i$ joins the vertices $\w(e_1)$ and $\w(e_2)$ of $\Wh_x(\w)$, where $e_1$ and $e_2$ are the two edges of $\ess_\w$ with $\iota(e_j)=x_i$.
\end{definition}

Note that the requirement that the multicycle $\w$ is an immersion implies that the endpoints of any edge of $\Wh_x(\w)$ are distinct.  However, each pair of vertices may be joined by many edges.

We can collect together all the Whitehead graphs at the vertices of $\Gamma$ into a global Whitehead graph for the pair $(\Gamma,\w)$.

\begin{definition}
The \emph{Whitehead graph} of the pair $(\Gamma,\w)$ is the disjoint union
\[
\Wh(\w):=\coprod_{x\in V}\Wh_x(\w)~.
\]
Note that $\Wh(\w)$ comes equipped with two additional structures:
\begin{enumerate}[(i)]
\item the components of $\Wh(\w)$ are naturally partitioned: two components are equivalent if they are both components of some $\Wh_x(\w)$;
\item the fixed-point free involution $e\mapsto\bar{e}$ on the edges of $\Gamma$ defines a fixed-point free involution on the vertices of $\Wh(\w)$ that extends to a bijection $i_e:\St(e)\to\St(\bar{e})$.
\end{enumerate}
We will always think of $\Wh(\w)$ as equipped with these extra structures.
\end{definition}

\begin{remark}\label{rem: Whitehead graphs determine pairs}
The partition on the components of $\Wh_x(\w)$ and the involutions $i_e$ are enough information to reconstruct the pair $(\Gamma,\w)$.
\end{remark}

Stallings studied morphisms of graphs by observing that they always factor as a composition of a sequence of \emph{folds} followed by an immersion.  Recall that a fold identifies a pair of edges $e_1,e_2$ with $\iota(e_1)=\iota(e_2)$.  It is therefore natural to study the effect that a fold has on Whitehead graphs.

\begin{definition}
Let $W$ be a graph and $v_1,v_2$ a pair of vertices.  (When we apply this, $W$ will be a disjoint union of Whitehead graphs.) A \emph{wedge} is a quotient map $W\to W'$ that identifies $v_1$ and $v_2$ and leaves the rest of $W$ unchanged.  We write $W\wedge_{v_1\sim v_2}$ for the quotient graph $W'$.  If $W=W_1\sqcup W_2$ with $v_i\in W_i$ then we write $W_1\wedge_{v_1\sim v_2} W_2$ for $W\wedge_{v_1\sim v_2}$.  The reverse move, which replaces $W\wedge_{v_1\sim v_2}$ by $W$, is called an \emph{unwedge}. 
\end{definition}

The following lemma shows that, at the level of Whitehead graphs, folds correspond to wedges.

\begin{figure}[h]
\begin{center}
 \centering \def\svgwidth{300pt}
 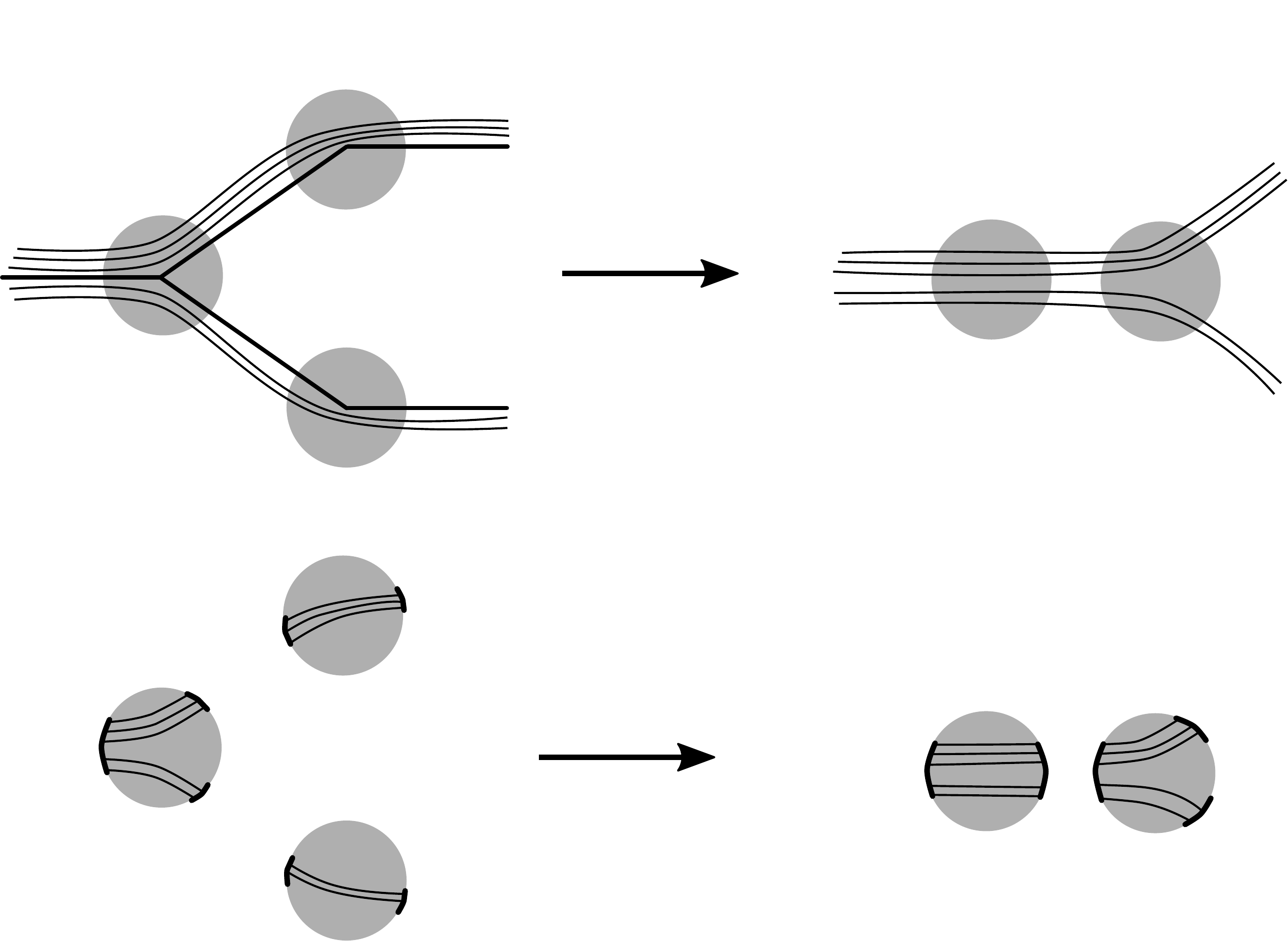
 \caption{The effect of a fold on Whitehead graphs.  Note that $\Wh_{x'}(\w')$ is obtained by wedging $\Wh_x(\w)$, and $\Wh_{y'}(\w')$ is obtained by wedging $\Wh_{y_1}(\w)\sqcup\Wh_{y_2}(\w)$.}
  \label{fig: Fold Whitehead}
\end{center}
\end{figure}

%Commented out the old version of the lemma.

\iffalse
\begin{lemma}
Let $f:(\Gamma,\w)\to(\Gamma',\w')$ be the morphism given by a fold $\Gamma\to\Gamma'$ which identifies two edges $e_1,e_2$ of $\Gamma$ with a common initial vertex. Suppose further that $f:\Gamma\to\Gamma'$ is a homotopy equivalence, i.e.\ $\tau(e_1)\neq\tau(e_2)$.  Let $x=\iota(e_1)=\iota(e_2)$, let $y_i=\iota(\bar{e}_i)$ and let $x',y'$ and $e'$ be the image vertices and edges in $\Gamma'$.  If $x$ is distinct from $y_1,y_2$ then:
\begin{enumerate}[(i)]
\item $\Wh_{x'}(\w')=\Wh_x(\w)\wedge_{e_1\sim e_2}$; and
\item if $y_1\neq y_2$ then $\Wh_{y'}(\w')=\Wh_{y_1}(\w)\wedge_{\bar{e}_1\sim \bar{e}_2}\Wh_{y_2}(\w)$.
\end{enumerate}
If $x=y_1$, say, then $x'=y'$ and, writing $y=y_2$, we have:
\begin{enumerate}[(i)]
\item $\Wh_{x'}(\w')=\Wh_x(\w)\wedge_{e_1\sim e_2}$; and
\item if $y_1\neq y_2$ then $\Wh_{y'}(\w')=\Wh_{y_1}(\w)\wedge_{\bar{e}_1\sim \bar{e}_2}\Wh_{y_2}(\w)$.
\end{enumerate}
\end{lemma}
\fi

\begin{lemma}\label{lem: Folds give splits}
Let $f:(\Gamma,\w)\to(\Gamma',\w')$ be the morphism given by a fold $\Gamma\to\Gamma'$, which identifies two edges $e_1,e_2$ of $\Gamma$ with a common initial vertex $x=\iota(e_1)=\iota(e_2)$ to an edge $e'$ of $\Gamma'$. Let $y_i=\iota(\bar{e}_i)$ and let $x'=\iota(e')$ and $y'=\tau(e')$.  Suppose further that $f:\Gamma\to\Gamma'$ is a homotopy equivalence, i.e.\ $y_1\neq y_2$.  Then
\[
\Wh_{x'}(\w')\cup\Wh_{y'}(\w')=(\Wh_{x}(\w)\wedge_{e_1\sim e_2})\cup (\Wh_{y_1}(\w)\wedge_{\bar{e}_1\sim\bar{e}_2}\Wh_{y_2}(\w))~.
\]
(Note that the unions in this expression may not be disjoint, since $x$ may equal $y_i$ for at most one $i$, in which case $x'$ also equals $y'$).  In particular, $\bar{e}'$ is a cut vertex of $\Wh_{y'}(\w')$.
\end{lemma}
\begin{proof}
This follows immediately from the definitions. (The case in which $x\neq y_1,y_2$ is illustrated in Figure \ref{fig: Fold Whitehead}).
\end{proof}

\begin{remark}
In the above lemma, the hypothesis that the map $f:(\Gamma,\w)\to(\Gamma',\w')$ is a  \emph{morphism of pairs} is essential: if the induced map $\ess_\w\to\ess_{\w'}$ were not an immersion, then after folding one would need to tighten $\w'$ to an immersion, which might destroy the cut point structure of the Whitehead graphs.
\end{remark}

\begin{remark}\label{rem: Wedges are injective on oriented edges}
In the setting of Lemma \ref{lem: Folds give splits}, for any vertex $v$ of $\Gamma$, the map $\Wh_v(\w)\to\Wh_{f(v)}(\w')$ induced by $f$ is injective on edges.
\end{remark}

In fact, the implication of Lemma \ref{lem: Folds give splits} can be reversed: if one of the Whitehead graphs has a cut vertex then we can unfold.

\begin{lemma}\label{lem: Splits give folds}
Let $(\Gamma',\w')$ be a graph pair, and suppose that some edge $\bar{e}'$ defines a cut vertex in $\Wh_{y'}(\w')$ (where $y'=\tau(e')$).  Then there is a graph pair $(\Gamma,\w)$ and a morphism of pairs defining a homotopy-equivalent fold $f:(\Gamma,\w)\to(\Gamma',\w')$ that identifies a pair of edges $e_1,e_2$ to $e'$.
\end{lemma}
\begin{proof}
The hypothesis tells us that $\Wh_{y'}(\w')=W_1\wedge_{\bar{e}_1\sim\bar{e}_2}W_2$, where $\bar{e}_i\in W_i$ is a vertex with image $\bar{e}'$ in the wedge.   Let $x'=\iota(e')$ (and note that $x'$ and $y'$ are not necessarily distinct).  We will define the pair $(\Gamma,\w)$ via its Whitehead graphs (appealing to Remark \ref{rem: Whitehead graphs determine pairs}). The proof divides into two similar cases, depending on whether or not $x'=y'$.

Suppose first that $x'\neq y'$.  For any vertex $z'$ of $\Gamma'$ not equal to $x'$ or $y'$, we take a vertex $z$ for $\Gamma$ with Whitehead graph isomorphic to $\Wh_{z'}(\w')$.  The remaining vertices of $\Gamma$ are denoted by $x,y_1,y_2$.   We define $\Wh_{y_i}(\w)$ to be $W_i$ for $i=1,2$.  Finally, $\Wh_x(\w)$ is defined so that
\[
\Wh_{x'}(\w')=\Wh_x(\w)\wedge_{e_1\sim e_2}~.
\]
That is, $\Wh_x(\w)$ is obtained from $\Wh_{x'}(\w')$ by dividing the vertex $e'$ into two vertices, $e_1,e_2$. The edges of $\Wh_{x}(\w)$ incident at $e_1$ and $e_2$ are defined so that they respect the natural bijections between the stars of the $e_i$ and the stars of the $\bar{e}_i$. There is then a natural lift of the bijections on stars in $\Wh(\w')$ to bijections on stars in $\Wh(\w)$, and this completes the construction of $(\Gamma,\w)$.

The case in which $x'=y'$ is similar.  Again, for any vertex $z'$ of $\Gamma'$ not equal to $x'$ or $y'$, we take a vertex $z$ for $\Gamma$ with Whitehead graph isomorphic to $\Wh_{z'}(\w')$.  The remaining vertices of $\Gamma$ are denoted by $x,y$.  Since $x'=y'$, the vertex $e'$ is contained in $\Wh_{y'}(\w')=W_1\wedge_{\bar{e}_1\sim \bar{e}_2} W_2$, and without loss of generality we may take $e'\in W_1$.  We now define $\Wh_y(\w)=W_2$, and $\Wh_x(\w)$ so that
\[
W_1 = \Wh_{x}(\w)\wedge_{e_1\sim e_2}~.
\]
As before, this means that $\Wh_{x}(\w)$ is obtained by dividing the vertex $e'$ into two vertices, $e_1,e_2$, and the edges of $\Wh_x(\w)$ incident at the $e_i$ are defined to respect the natural bijections between the stars of the $e_i$ and the stars of the $\bar{e}_i$.  Again, this completes the construction of $(\Gamma,\w)$.

In either case, identifying $e_1$ and $e_2$ defines a fold $(\Gamma,\w)\to (\Gamma',\w')$.  Since $\tau(e_2)\neq \tau(e_1)$, the fold is a homotopy equivalence.  Note also that, by construction, the fold is a morphism of pairs.
 \end{proof}

Whitehead introduced Whitehead graphs to recognize basis elements of free groups and, more generally, free splittings. (More generally still, Whitehead gave an algorithm to find the shortest element in an orbit of the automorphism group.)   We will use Whitehead graphs to recognize \emph{(weakly) irreducible} pairs.

\begin{definition}\label{defn: Irreducible pair}
Consider a graph pair $(\Gamma,\w)$ with $\Gamma$ a finite connected graph.  By Grushko's theorem, $\pi_1\Gamma$ splits canonically as
\[
G_1*\ldots*G_k*F
\]
where: for each $j=1,\ldots,k$, there is an index set $I_j$ so that $w_i$ is conjugate into $G_j$ for all $i\in I_j$; each $G_j$ does not split freely relative to the set $\{w_i\mid i\in I_j\}$; and no $w_i$ is conjugate into $F$.  A factor $G_j$ is called \emph{cyclic} if  $I_j$ is a singleton $\{i\}$ and, up to conjugacy, $w_{i}$ generators $G_j$.

The pair $(\Gamma,\w)$ is called \emph{weakly irreducible} if there are  no cyclic factors; otherwise it is called \emph{strongly reducible}. The pair  is called {reducible} if it is weakly irreducible and $k=1$; otherwise it is called \emph{reducible}.  When $\Gamma$ is disconnected, the pair $(\Gamma,\w)$ is called (weakly) irreducible or (strongly) reducible if and only if each component has that property.
\end{definition}

The point of the above definition is that the pair $(\Gamma,\w)$ is reducible if and only if the fundamental group of the double $D(\w)$, obtained as a graph of spaces with two vertex spaces homeomorphic to $\Gamma$ and edge maps given by $\w$, admits a non-trivial free splitting.  Cyclic factors are relevant because they give rise to $\mathbb{Z}$ factors of the double. Even a single cyclic factor is reducible, since $\Z$ splits as an HNN extension of the trivial group.

The following lemma is the key result for recognizing reducible pairs.  It is quite standard, but we give a proof using folds and wedges as a sample application of the above ideas.

\begin{lemma}[Whitehead]\label{lem: Whitehead's lemma}
If $(\Gamma,\w)$ is reducible then there is a vertex $x$ of $\Gamma$ so that one of the following holds:
\begin{enumerate}[(i)]
\item $\Wh_x(\w)$ is disconnected;
\item $\Wh_x(\w)$ has a leaf, i.e.\ a vertex of valence 1; or
\item $\Wh_x(\w)$ has a cut vertex, i.e.\ a vertex $e$ so that $\Wh_x(\w)\smallsetminus\{e\}$ is disconnected.
\end{enumerate}
\end{lemma}
\begin{proof}
The case in which $\pi_1\Gamma$ is cyclic and $\w:S^1\to\Gamma$ is a $\pi_1$-isomorphism is easy and left as an exercise.  Suppose therefore that $\pi_1\Gamma$ admits a free splitting relative to $\w$.  It follows that there is a morphism of graphs $f:(\Gamma',\w')\to(\Gamma,\w)$ which is a homotopy equivalence, so that $\Gamma'$ has a vertex $x'$ with $\Wh_{x'}(\w')$ disconnected.  The morphism $f$ now factors as a sequence of homotopy-equivalent folds; in particular, whenever $e_1$ and $e_2$ with $\iota(e_1)=\iota(e_2)$ are identified, we have $\tau({e}_1)\neq\tau({e}_2)$.  Consider the final such fold, which identifies a pair of distinct vertices $y_1,y_2$ to a vertex $y$.  Lemma \ref{lem: Folds give splits} implies that $\Wh_y(\w)$ has a cut vertex.
\end{proof}

Motivated by Whitehead's lemma, we call a Whitehead graph $\Wh_x(\w)$ \emph{reducible} if it satisfies one of the three conclusions of the lemma; otherwise, we call $\Wh_x(\w)$ \emph{irreducible}.  We call the pair $(\Gamma,\w)$ \emph{locally irreducible} if, for every vertex $x$ of $\Gamma$, the Whitehead graph $\Wh_x(\w)$ is irreducible.  Whitehead's lemma therefore says that a locally irreducible pair is irreducible.

The converse to this statement is not quite true -- there are irreducible pairs that are not locally irreducible.  To construct an example, take an irreducible pair and apply one fold.  We therefore must allow ourselves to unfold in order to prove a converse to Whitehead's lemma.

\begin{lemma}[Converse to Whitehead's lemma]\label{lem: Converse to Whitehead}
If a pair $(\Gamma,\w)$ is irreducible then there is a locally irreducible pair $(\Gamma',\w')$ and a map of pairs $(\Gamma',\w')\to(\Gamma,\w)$ which is a homotopy equivalence.
\end{lemma}
\begin{proof}
If some vertex of $\Gamma$ has either a disconnected Whitehead graph or a leaf then $(\Gamma,\w)$ is reducible.  Suppose therefore that there is a vertex $y$ so that $\Wh_y(\w)$ has a cut vertex $\bar{e}$.  By Lemma \ref{lem: Splits give folds}, there is a fold $(\Gamma',\w')\to(\Gamma,\w)$ so that the edge $e$ is unfolded to a pair of edges $e'_1,e'_2$.

Note that: the number of edges of $\Wh(\w')$ is equal to the number of edges of $\Wh(\w)$; the number of vertices of $\Wh(\w')$ is greater than the number of vertices of $\Wh(\w)$; for each vertex $z'$ of $\Gamma'$, the Whitehead graph $\Wh_{z'}(w')$ is connected without leaves.  In particular, the number of vertices of $\Wh(\w')$ is at most the number of edges of $\Wh(\w)$. It follows that only finitely many unfoldings of this form can be performed.

When no further unfoldings can be performed, the final pair $(\Gamma',\w')$ is locally irreducible, as claimed.
\end{proof}

The final lemma of this section shows that we can use locally irreducible $\partial$-immersions to recognize weakly irreducible immersions.

\begin{lemma}\label{lem: Locally irreducible d-immersions}
If $(\Lambda,\uu)\to(\Delta,\vv)\immerses (\Gamma,\w)$ is a $\partial$-immersion and $(\Lambda,\uu)$ is locally irreducible then the pair $(\Delta,\vv)$ is weakly irreducible.
\end{lemma}
\begin{proof}
By repeatedly applying Lemma \ref{lem: Splits give folds} as in the proof of Lemma \ref{lem: Converse to Whitehead}, there is a homotopy-equivalent morphism of pairs $(\Delta',\vv')\to(\Delta,\vv)$ so that every Whitehead graph of $(\Delta',\vv')$ has no cut vertices.  Since $(\Lambda,\uu)$ is locally irreducible, the morphism $(\Lambda,\uu)\to(\Delta,\vv)$ lifts to a morphism $(\Lambda,\uu)\to(\Delta',\vv')$; note that this map is bijective on edges of Whitehead graphs and surjective on vertices.  If some component of a Whitehead graph of $(\Delta',\vv')$ had at most one edge, so would a component of a Whitehead graph of $(\Lambda,\uu)$ that mapped to it, contradicting the hypothesis that $(\Lambda,\uu)$ is locally irreducible.
\end{proof}

\section{Admissible $\partial$-immersions}\label{sec: Admissible pairs}

Consider an irreducible graph pair $(\Gamma,\w)$ as above.  In this section, we will study $\partial$-immersions
\[
(\Lambda,\uu)\to (\Delta,\vv)\immerses(\Gamma,\w)
\]
where $(\Lambda,\uu)$ is a locally irreducible pair.  We will impose one additional condition on our $\partial$-immersions.

\begin{definition}
A map of graph pairs $(\Lambda,\uu)\to (\Gamma,\w)$ is called \emph{admissible} if there is a positive integer $n=n(\uu)$ so that every point in $\ess_\w$ has exactly $n$ preimages in $\ess_\uu$.  A $\partial$-immersion $(\Lambda,\uu)\to(\Delta,\vv)\immerses (\Gamma,\w)$ is called \emph{admissible} if the composition  $(\Lambda,\uu)\to (\Gamma,\w)$ is admissible.
\end{definition}

Calegari uses the term `admissible' similarly for maps of surfaces with boundary \cite[p.\ 37]{calegari_scl_2009}.  Note that, in his context, the integer $n$ counts preimages with a sign determined by orientation, whereas in our context, the count is unsigned.

%Commented out the definition for group pairs.

\iffalse
There is a corresponding notion for group pairs.  A map of group pairs $(f,\phi):(G,\curlyB)\to (F,\curlyA)$ is \emph{admissible} if there is a positive integer $n=n(\uu)$ so that
\[
n=\sum_{\phi(b)=a}|\Stab_F(a) : f(\Stab_G(b)) |
\]
for each $a\in\curlyA$.
\fi

We next write down a finite set of pieces, from which $\partial$-immersions of locally irreducible pairs can be constructed.

\begin{definition}\label{defn: P}
Let $\curlyW$ be the set of components of $\Wh(\w)$.  The set  $\curlyP$ of \emph{pieces (over $\Wh(\w)$)} consists of all pairs of maps of graphs (up to graph isomorphism)
\[
P\to V\into W
\]
such that $P$ is a disjoint union of irreducible graphs, the map $P\to V$ is bijective on edges, the map $V\into W$ is injective, and $W\in\curlyW$.  When $P\to V\into W$ is an element of $\curlyP$, we will often abuse notation and write $P\in\curlyP$, since the map $P\to W$ determines $V$.
\end{definition}

\begin{remark}
Note that $\curlyP$ is finite.   This trivial observation is of crucial importance.
\end{remark}

We will study admissible $\partial$-immersions of locally irreducible pairs by looking at how they are constructed from the pieces $\curlyP$.  The involutions $i_e$ on the stars of vertices of $\Wh(\w)$ define relations on the elements of $\curlyP$, as follows.

\begin{definition}
Consider
\[
P\to U\into W_1~,~Q\to V\to W_2
\]
elements of $\curlyP$.    Suppose that $e$ is a vertex of $U\subseteq W_1$ and that $\bar{e}$ is a vertex of $V\subseteq W_2$.  Let $e_1,\ldots,e_m$ be the set of vertices of $P$ that map to $e\in U$, and let $\bar{e}_1,\ldots,\bar{e}_n$ be the set of vertices of $Q$ that map to $\bar{e}\in V$.   We write
\[
P\leftrightarrow_{e} Q 
\]
if:
\begin{enumerate}[(i)]
\item $m=n$, and
\item up to reordering of indices, $i_e$ restricts to bijections $\St_{P}(e_j)\to\St_Q(\bar{e}_j)$ for all $j$.
\end{enumerate}
\end{definition}

The relation $P\leftrightarrow_{e} Q$ can be interpreted in terms of Manning's \emph{splicing} operation \cite{manning_virtually_2010}.  It says that $P$ and $Q$ can be \emph{spliced} at the sets of vertices $\{e_1,\ldots,e_n\}$ and $\{\bar{e}_1,\ldots,\bar{e}_n\}$, and that $U$ and $V$ can be spliced at $e$ and $\bar{e}$.

To record how the elements of $\curlyP$ are glued together, we introduce \emph{$\curlyP$-stars}.

\begin{definition}
A \emph{$\curlyP$-star} $\sigma$ consists of the following data:
\begin{enumerate}[(i)]
\item a piece $P\to U\into W$ in $\curlyP$;
\item for each vertex $e_i$ of $U$, a choice of piece $Q_i\to V_i\into W_i$ in $\curlyP$ and a vertex $\bar{e}_i$ of $V_i$ so that $P\leftrightarrow_{e_i} Q_i$.
\end{enumerate}
We write $\sigma(*)=P$ and, for each vertex $e_i$ of $P$, we write $\sigma(e_i)=Q_i$.
\end{definition}

Let $\curlyS\equiv\curlyS(\curlyP)$ be the (finite) set of all $\curlyP$-stars.  Let $V_\curlyP=\R^{\curlyS}$ and let $V_\curlyP^+\subseteq V_\curlyP$ be the non-negative orthant.  An admissible $\partial$-immersion
\[
(\Lambda,\uu)\to(\Delta,\vv)\looparrowright (\Gamma,\w)
\]
of a locally irreducible pair $(\Lambda,\uu)$ defines an integer vector $\pi(\uu)\in V_\curlyP^+$ in a natural way, as follows.   For each vertex $x$ of $\Delta$, let $x'_1,\ldots,x'_k$ be the preimages of $x$ in $\Lambda$ and let $x_0$ be the image vertex in $\Gamma$.
The piece $P(x)\in\curlyP$ is then defined to be
\[
\coprod_{i=1}^k \Wh_{x'_i}(\uu)\to \Wh_x(\vv)\into \Wh_{x_0}(\w)~.
\]
If $e$ is an edge of $\Delta$ with $\iota(e)=x$ and $\tau(e)=y$, then $P(x)\leftrightarrow_e P(y)$. Therefore, for each vertex $x$ of $\Delta$, we can define $\sigma_x$ to be the corresponding $\curlyP$-star associated to the labels of the neighbouring vertices: 
\begin{enumerate}[(i)]
\item $\sigma_x(*)=P(x)$; and,
\item for each $e_i\in\St_\Lambda(x)$, $\sigma_x(e_i)=P(\tau(e_i))$.
\end{enumerate}
Now define $\pi(\uu)$ to be the vector $\underline{p}\in V_\curlyP^+$ so that
\[
p_\sigma=\#\{x\in V(\Delta)\mid \sigma_x=\sigma\}
\]
for each $\sigma\in\curlyS$.

The image of $\pi$ is not arbitrary: in fact, it is precisely the set of integer points of a certain cone $\cone\subseteq V_\curlyP^+$.  We will describe this cone using systems of equations: the gluing equations and the admissibility equations.  We start with the gluing equations.  A non-negative integer vector that satisfies the gluing equations necessarily comes from a $\partial$-immersion of a locally irreducible pair.

\begin{definition}
We write $\x=(x_\sigma)_{\sigma\in\curlyS}$ for an element of $V_\curlyP$.  For each pair of pieces $P,Q\in\curlyP$ and edge $e$ satisfying $P\leftrightarrow_{e}  Q$ we have the \emph{gluing equation}
\[
\sum_{\substack{\sigma(*)=P \\ \sigma(e)=Q}} x_\sigma=\sum_{\substack{\sigma(*)=Q \\ \sigma(\bar{e})=P}} x_\sigma
\] 
where each sum is taken over all $\curlyP$-stars $\sigma$ satisfying the conditions.
\end{definition}

We next describe the admissibility equations, which force any $\curlyP$-pair that defines a vector to be admissible.

\begin{definition}
Let $\epsilon$ be an edge of $\Wh(\w)$.  For a piece
\[
P\to V\into W
\]
in $\curlyP$, set $\delta_\epsilon(P)$ to be the number of preimages of $\epsilon$ in $P$.  (Note that this is either 0 or 1, by definition.)  We now define a linear map $n_\epsilon:V_\curlyP\to\R$ by setting 
\[
n_\epsilon(\x)=\sum_{\sigma\in\curlyS} x_\sigma\delta_\epsilon(\sigma(*))~.
\]
The \emph{admissibility equations} assert that $n_\epsilon(\x)=n_{\epsilon'}(\x)$  for all edges $\epsilon$ and $\epsilon'$ in $\Wh(\w)$.
\end{definition}

The cone $\cone\subseteq V_\curlyP^+$ is now defined to be the subset of $V_\curlyP^+$ that satisfies the gluing equations and the admissibility equations.

\begin{lemma}\label{lem: Vectors and pairs}
An integer vector $\x\in V_\curlyP^+\smallsetminus 0$ is the image of an admissible $\partial$-immersion from a locally irreducible pair under $\pi$ if and only if it is in $\cone$.
\end{lemma}
\begin{proof}
Let $(\Lambda,\uu)$ be locally irreducible and
\[
(\Lambda,\uu)\to(\Delta,\vv)\immerses(\Gamma,\w)
\]
a $\partial$-immersion.  First we show that $\pi(\uu)$ satisfies the gluing equations.  Indeed, the expression in the gluing equations is just two different ways of evaluating the number of edges $e$ of $\Delta$ with $P(\iota(e))=P$ and $P(\tau(e))=Q$.  The admissibility equations are satisfied since each $n_\epsilon$ evaluates to $n(\uu)$.

Conversely, given an integer vector $\x\in\cone$, we need to construct an admissible $\partial$-immersion 
\[
(\Lambda,\uu)\to(\Delta,\vv)\immerses(\Gamma,\w)
\]
with $(\Lambda,\uu)$ locally irreducible.  By Remark \ref{rem: Whitehead graphs determine pairs}, it is enough to describe
\[
\Wh(\uu)\to\Wh(\vv)\into\Wh(\w)
\]
together with their pairings on stars of vertices.  For each star $\sigma$, $\Wh(\uu)\to\Wh(\vv)\into\Wh(\w)$ contains $x_\sigma$ copies of the piece $\sigma(*)$.  This determines the graphs and maps $\Wh(\uu)\to\Wh(\vv)\into\Wh(\w)$; it remains to determine the pairings. Consider the pieces
\[
P\to U\into W_1~,~Q\to V\into W_2
\]
and suppose that $e$ is an edge of $U$ and $\bar{e}$ is an edge of $V$.  The gluing equations imply that here is a bijection between the number of $\curlyP$-stars $\sigma$ so that $\sigma(*)=P$ and $\sigma(e)=Q$ and the number of $\curlyP$-stars $\sigma$ so that $\sigma(*)=Q$ and $\sigma(\bar{e})=P$ which satisfy the condition that the bijection $i_e:\St_{W_1}(e)\to\St_{W_2}(\bar{e})$ restricts to a bijection $\St_U(e)\to\St_V(\bar{e})$, and thence to bijections of the stars of the preimages in $P$ and $Q$.  Choosing a bijection between these $\curlyP$-stars then determines the required bijection between vertices of the copies of $P$ and $Q$ in these $\curlyP$-stars, and the bijections between stars are then determined by the relation $\leftrightarrow_e$.

By construction, $(\Lambda,\uu)$ is locally irreducible and
\[
(\Lambda,\uu)\to (\Delta,\vv)\immerses(\Gamma,\w)
\]
is a $\partial$-immersion.  Finally, the admissibility equations immediately imply that this $\partial$-immersion is admissible.
\end{proof}

Thus, we have seen that admissible $\partial$-immersions of locally irreducible pairs correspond naturally to non-zero integer vectors in $\cone$ or, equivalently, to rational points in the projectivization  $\Proj(\cone)$.  Motivated by Calegari's work on stable commutator length (see \cite{calegari_stable_2009} and \cite{calegari_scl_2009}, and also \cite{brady_turn_2011}), we will study these rational points via rational functions on $\Proj(\cone)$.

\section{The rationality theorem}\label{sec: Rationality}

We start by writing down two natural linear maps on $\cone$.   For an admissible $\partial$-immersion
\[
(\Lambda,\uu)\to(\Delta,\vv)\immerses (\Gamma,\w) ~,
\]
the corresponding linear maps are (minus) the Euler characteristic of $\Lambda$, and the degree $n(\uu)$ with which $\ess_\uu$ covers $\ess_\w$.  The key observation is that both of these can be computed from the vector $\pi(\uu)$.

First, the admissibility equations imply that the linear map $n_\epsilon$ is independent of $\epsilon$.  We therefore write $n=n_\epsilon$, evidently a linear map which is non-zero on $\cone\smallsetminus 0$.

Second, for a piece
\[
P\to V\into W
\]
in  $\curlyP$, we let $\mu(V)$ denote the number of connected components of $V$, and let $\nu(V)$  denotes the number of vertices of $V$. We then define $\chi_-:V_\curlyP\to \R$ by
\[
\chi_-(\x)=\sum_{\sigma\in\curlyS}x_\sigma\left(\frac{1}{2}\nu(\sigma(*))-\mu(\sigma(*))\right)~.
\]
It's well known that the Euler characteristic of a graph can be computed as the sum over the vertices of one minus half the valence, and from this we see that, for an admissible $\partial$-immersion
\[
(\Lambda,\uu)\to(\Delta,\vv)\immerses(\Gamma,\w)
\]
of a locally irreducible pair $(\Lambda,\uu)$, we have $\chi_-\circ\pi(\uu)=-\chi(\Delta)$. 

\begin{definition}
Since $\chi_-/n$ is a quotient of two linear maps on $\cone\smallsetminus 0$ and the denominator is non-zero, it yields a well defined function on the projectivization $\Proj(\cone)$.  We call this function
\[
\rho_\curlyP=\frac{\chi_-}{n}
\]
the \emph{projective $\curlyP$-rank} function on $\Proj(\cone)$.
\end{definition}

In analogy with stable commutator length (see \cite{calegari_scl_2009}), we may use the rational function $\rho_\curlyP$ to define an invariant of a multicycle $\w$ in a graph $\Gamma$.

\begin{definition}
The \emph{maximal $\curlyP$-rank} of a pair $(\Gamma,\w)$ is denoted by $\rho^+_\curlyP(\w)$ and defined to be
\[
\rho^+_\curlyP(\w):=\max_{[\x]\in\Proj(\cone)}\rho_\curlyP[\x]~.
\]
Note that this maximum is indeed realized, since $\Proj(\cone)$ is compact.  Similarly, the \emph{minimal} $\curlyP$-rank, $\rho^-_\curlyP(\w)$, is defined to the minimum of $\rho_\curlyP$ over the same domain.
\end{definition}

Since $\Proj(\cone)$ is compact, the maximal and minimal $\curlyP$-ranks are certainly attained as long as $\Proj(\cone)$ is non-empty (i.e.\ as long as $\cone$ is non-zero).  In fact, since $\rho_\curlyP$ is a quotient of linear maps, the maximal and minimal $\curlyP$-ranks are attained on rational points of $\Proj(\cone)$, and hence are realized by admissible $\partial$-immersions.

\begin{theorem}\label{thm: Rationality theorem}
If $\cone\neq 0$, then the maximal and minimal $\curlyP$-ranks are realized by admissible $\partial$-immersions of locally irreducible pairs; that is, there exist locally irreducible pairs $(\Lambda_\pm,\uu_\pm)$ and admissible $\partial$-immersions
\[
(\Lambda_\pm,\uu_\pm)\to(\Delta_\pm,\vv_\pm)\immerses(\Gamma,\w)
\]
so that $\rho_\curlyP\circ\pi(\uu_\pm)=\rho^\pm_\curlyP(\w)$.  In particular, $\rho^\pm_\curlyP(\w)$ are positive rational numbers.
\end{theorem}
\begin{proof}
We prove the result for the maximal $\curlyP$-rank; the proof for the minimal $\curlyP$-rank is identical.  If $\cone\neq 0$ then the projectivization $\Proj(\cone)$ is non-empty.   From the definition of $\rho^+_\curlyP$, we may normalize and restrict our attention to the rational polytope $n^{-1}(1)$, so
\[
\rho^+_\curlyP(\w)=\max_{n(\x)=1}\chi_-(\x)~.
\]
But $\chi_-$ is linear, and so attains its maximum on a vertex $\x_0$ of $n^{-1}(1)$.  Since $\x_0$ is rational and $\rho^+_\curlyP$ is a projective function, there is some integer vector $\x_1$, a multiple of $\x_0$, on which $\rho^+_\curlyP$ attains its maximum.  Since $\x_1\in\cone$ is an integer vector,  it is equal to $\pi(\uu)$ for some admissible $\curlyP$-pair $(\Lambda,\uu)$, which therefore realizes $\rho^+_\curlyP$, as required.
\end{proof}

An admissible $\partial$-immersion
\[
(\Lambda,\uu)\to(\Delta,\vv)\immerses(\Gamma,\w)~,
\]
of a locally irreducible pair $(\Lambda,\uu)$ for which $\rho_\curlyP\circ\pi(\uu)=\rho^+_\curlyP(\w)$ is called \emph{maximal}. (Similarly, if $\rho_\curlyP\circ\pi(\uu)=\rho^-_\curlyP(\w)$ then the $\partial$-immersion is called \emph{minimal}.)

\section{Maximal $\curlyP$-rank and surfaces}\label{sec: Surfaces}

Our results so far imply that every irreducible pair admits a maximal $\partial$-immersion.

\begin{lemma}\label{lem: Maximal pairs exist}
If $(\Gamma,\w)$ is irreducible then there exists a maximal, admissible $\partial$-immersion
\[
(\Lambda,\uu)\to(\Delta,\uu)\immerses(\Gamma,\w)
\]
for a locally irreducible pair $(\Lambda,\uu)$.
\end{lemma}
\begin{proof}
Since $(\Gamma,\w)$ is irreducible, Lemma \ref{lem: Converse to Whitehead} guarantees a locally irreducible pair $(\Gamma',\w')\to(\Gamma,\w)$.  The map $(\Gamma',\w')\to(\Gamma,\w)$ consists of a $\pi_1$-isomorphism and a  homeomorphism $\ess_{\w'}\to\ess_\w$, so is certainly admissible and essential.  In particular,
\[
(\Gamma',\w')\to(\Gamma,\w)\stackrel{\cong}{\to} (\Gamma,\w)
\]
is an admissible $\partial$-immersion of a locally irreducible pair, so $\pi(\w')\in\cone$ and $\cone\neq 0$.  Theorem \ref{thm: Rationality theorem} now implies that a maximal $\partial$-immersion exists.
\end{proof}

In this section, we shall use the relative JSJ decomposition together with the results of \cite{wilton_one-ended_2011} to show that maximal $\partial$-immersions are closely related to surfaces. 

\begin{definition}
A group pair is said to be of \emph{surface type} if it arises as the fundamental group of  a space pair $(\Sigma,\partial\Sigma)$, where $\Sigma$ is a compact surface with boundary.  It is said to be of \emph{weak surface type} if it is a free product of pairs of surface type.   A graph pair $(\Gamma,\w)$ is of \emph{(weak) surface type} if the corresponding group pair is of (weak) surface type.
\end{definition}

A theorem of Culler \cite{culler_using_1981} shows that any pair $(\Gamma,\w)$ of surface type can be unfolded to a \emph{fatgraph} $(\Gamma',\w')$ -- a graph pair in which every Whitehead graph is a cycle. One may therefore equivalently think of pairs of surface type as given by fatgraphs.  Likewise, a pair of weak surface type can be unfolded to a graph pair in which every Whitehead graph is a disjoint union of cycles.
 
Fundamental groups of pairs of surface type can typically be decomposed as graphs of groups in many ways.  In order to discuss this, we introduce some terminology for graph-of-groups decompositions of pairs. 

\begin{definition}\label{defn: Decomposition of group pair}
Let $(G,\curlyA)$ be a group pair.  A \emph{decomposition} of $(G,\curlyA)$ is a graph of groups $\curlyG$ with fundamental group $G$ such that, for every $a\in\curlyA$, the stabilizer $\Stab_G(a)$ is conjugate into a vertex group of $\curlyG$. 

Let $v$ be a vertex of $\curlyG$, and fix a pre-image $\tilde{v}$ of $v$ in the Bass--Serre tree $T$.  Let $G_{\tilde{v}}$ be the stabilizer of $\tilde{v}$.  Set
\[
\curlyA_{\tilde{v}}=\{a\in \curlyA\mid \Stab_{G_{\tilde{v}}}(a)\neq 1\}
\]
and let $\curlyB_{\tilde{v}}$ be the set of edges of $T$ incident at $\tilde{v}$.  The \emph{induced pair} at $v$ is defined to be $(G_{\tilde{v}},\curlyA_{\tilde{v}}\sqcup\curlyB_{\tilde{v}})$, which is defined up to conjugacy in $G$.  The vertex $v$ is called \emph{peripheral} if $\curlyA_{\tilde{v}}$ is non-empty.

If every edge group of $\curlyG$ is cyclic then $\curlyG$ is said to be a \emph{cyclic} decomposition of $(G,\curlyA)$.   As usual, the graph of groups $\curlyG$ is called \emph{trivial} if $G$ is the stabilizer of some vertex of the Bass--Serre tree. 
\end{definition}

We will only be concerned with cyclic decompositions of graph pairs $(\Gamma,\w)$, with $F=\pi_1\Gamma$.  We will abuse notation and write the corresponding group pair as $(F,\w)$. 

Pairs of surface type can be contrasted with rigid pairs, which only have trivial decompositions.

\begin{definition}
An irreducible graph pair $(\Gamma,\w)$ is \emph{rigid} if every cyclic decomposition of $(\Gamma,\w)$ is trivial, and if $(\Gamma,\w)$ is not of surface type.  (This last requirement is to rule out the pair of pants, which is of surface type but admits no cyclic decompositions.)   A group pair $(F,\w)$ is \emph{rigid} if some (any) corresponding graph pair is rigid.
\end{definition}

%Commented out definitions for space pairs. Plan to add more about them before the proofs of the lemmas.

\iffalse

\begin{definition}\label{defn: Decomposition of space pair}
Let $(X,\w)$ be a space pair.  A \emph{decomposition} of $(X,\w)$ is a graph of spaces $\curlyX$ homotopy equivalent to $X$ such that, for every component $Y_i$ of the domain of $\w$, $\w(Y_i)$ is contained in a vertex space of $\curlyX$.  

Let $v$ be a vertex of the underlying graph of $\curlyX$, let $I_v$ be the set of indices $i$ so that $\w(Y_i)$ is contained in $\curlyX_v$ and let $\{e_j\}$ be the set of edges of the underlying graph of $\curlyX$ incident at $v$. If we set $\w_v$ to be the the restriction of $\w$ to the union of components $\bigcup_{i\in I_v }Y_i$, and $\partial_v$ to be the coproduct of the edge maps $\curlyX_{e_j}\to X_v$, then the \emph{induced pair} at $v$ is defined to be
\[
(\curlyX_v,\w_v\sqcup\partial_v)~.
\]
That is, the induced space pair at a vertex consists of the vertex space, together with incident edge spaces together with any components of the original peripheral structure that lie in the vertex space.

If every edge space of $\curlyX$ is  a circle then $\curlyX$ is said to be a \emph{cyclic} decomposition of $(X,\w)$.
\end{definition}

\fi

The main theorem of this section is phrased in terms of the \emph{relative JSJ decomposition} of the group pair $(F,\w)$.  This is a canonical decomposition of the pair $(F,\w)$, which in a sense encodes all cyclic decompositions.  The absolute version of this decomposition was described in the hyperbolic case by Bowditch \cite{bowditch_cut_1998}; the relative version in the free case was described by Cashen \cite{cashen_splitting_2016}. See also the work of Guirardel and Levitt, who explain how to construct this JSJ decomposition as a tree of cylinders \cite{guirardel_JSJ_2017}. 

\begin{theorem}[Relative JSJ decomposition]\label{thm: JSJ}
Let $(F,\w)$ be an irreducible group pair.  There is a canonical cyclic decomposition $\curlyG$ for $F$ with the following properties.
\begin{enumerate}[(i)]
\item The underlying graph of $\curlyG$ has three kinds of vertices -- rigid, surface and cyclic -- such that:
\begin{enumerate}
\item if a vertex $v$ is of rigid type then the induced pair at $v$ is a rigid group pair;
\item if a vertex $v$ is of surface type then the induced pair at $v$ is of surface type;
\item if a vertex $v$ is cyclic then the vertex group $\curlyG_v$ is (infinite) cyclic.
\end{enumerate}
\item The underlying graph of $\curlyG$ is bipartite, with red vertices cyclic and green vertices either rigid or surface.  In particular, every edge adjoins exactly one cyclic vertex.
\item Every peripheral subgroup $\langle w_i\rangle$ is conjugate into a unique cyclic vertex group. (These cyclic vertices are called \emph{peripheral}.)
\end{enumerate}
\end{theorem}

The decomposition $\curlyG$ guaranteed by the theorem is called the \emph{relative JSJ decomposition} of the pair $(F,\w)$.  For an irreducible graph pair $(\Gamma,\w)$, we will refer to the disjoint union of the relative JSJ decompositions of the fundamental groups of the components as the \emph{relative JSJ decomposition} of the pair $(\Gamma,\w)$. 

We are now ready to state the main theorem of this section, which describes the relative JSJ decompositions of maximal $\partial$-immersions.

\begin{theorem}\label{thm: Maximal pairs and surfaces}
If $(\Lambda,\uu)$ is locally irreducible and
\[
(\Lambda,\uu)\to(\Delta,\vv)\immerses(\Gamma,\w)
\]
is a maximal, admissible $\partial$-immersion then  for each irreducible free factor $(\pi_1\Delta_i,\vv_i)$ of the corresponding group pair $(\pi_1\Delta,\vv)$,  the relative JSJ decomposition of $(\pi_1\Delta_i,\vv_i)$ has no rigid vertices.  Furthermore, if there is such a maximal $\partial$-immersion, then there is a maximal, admissible $\partial$-immersion so that $(\Delta,\vv)$ is of weak surface type.
\end{theorem}

The proof is based on ideas from \cite{wilton_one-ended_2011};  the main technical result of that paper is as follows \cite[Theorem 8]{wilton_one-ended_2011}.

\begin{theorem}\label{thm: Local theorem}
If $(\Gamma,\w)$ is a rigid graph pair then there is a finite-sheeted cover $(\wh{\Gamma},\hat{\w})\to (\Gamma,\w)$ such that, whenever a finite-sheeted cover $(\Gamma',\w')\to(\Gamma,\w)$ factors through $(\wh{\Gamma},\hat{\w})$, the pair $(\Gamma,\w'\smallsetminus w'_i)$ is irreducible for any component $w'_i$ of $\w'$.
\end{theorem}

We will also need a relative analogue of Shenitzer's lemma --  see, for instance, \cite[Corollary 1.1]{touikan_one-endedness_2015} -- which we state here in the terminology of this paper.

\begin{lemma}[Relative Shenitzer's lemma]\label{lem: Relative Shenitzer}
Consider a decomposition $\curlyG$ of a group pair $(F,\w)$.  If the induced pair at every vertex of $\curlyG$ is irreducible then the pair $(F,\w)$ is irreducible.
\end{lemma}

We now assemble the lemmas that we will need to prove Theorem \ref{thm: Maximal pairs and surfaces}.  The first shows how to use a rigid vertex to increase irreducible rank.  Its proof is illustrated in Figure \ref{fig: Rigid vertex}.

\begin{figure}[h]
\begin{center}
 \centering \def\svgwidth{400pt}
 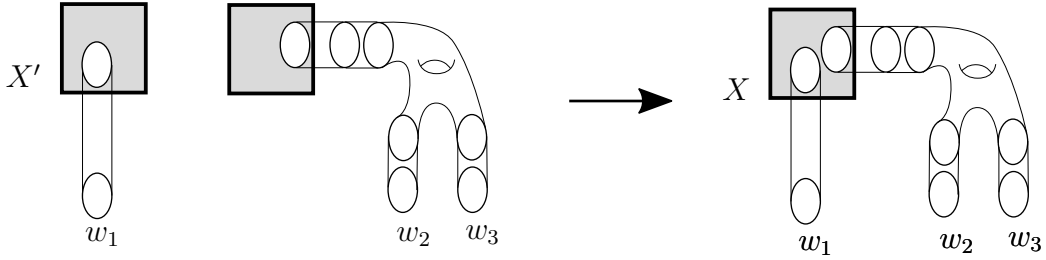
 \caption{The graph of spaces $X$ has one rigid vertex, one surface vertex, one non-peripheral cyclic vertex and three peripheral cyclic vertices.  By taking multiple copies of $X$ and deleting complementary components of the rigid vertices, we construct a new graph of spaces $X'$ with greater projective $\curlyP$-rank.}
  \label{fig: Rigid vertex}
\end{center}
\end{figure}

\begin{lemma}\label{lem: Rigid implies non-maximal}
Let $(\Gamma,\w)$ be an irreducible graph pair.  If the relative JSJ of $(\Gamma,\w)$  has a rigid vertex then there is a locally irreducible pair $(\Gamma',\w')$ and an admissible, essential map $(\Gamma',\w')\to(\Gamma,\w)$ with
\[
\frac{\chi_-(\Gamma')}{n(\w')}>\chi_-(\Gamma)~.
\]
\end{lemma}
\begin{proof}
Consider the relative JSJ decomposition $\curlyG$ of the pair $(\Gamma,\w)$.  By Theorem \ref{thm: Local theorem} and Marshall Hall's theorem \cite{hall_jr_subgroups_1949}, after replacing $(\Gamma,\w)$ with a finite-sheeted cover we may assume that every rigid vertex $\curlyG_u$ has the property guaranteed by Theorem \ref{thm: Local theorem}.  Note that, if $\nu(v)$ denotes the valence of the vertex $v$, then every rigid vertex $u$ has $\nu(u)>1$.

To construct $\Gamma'$, we realize the relative JSJ of $(\Gamma,\w)$ as a graph of spaces $X$. We may take the vertex spaces of $X$ to be graphs and the edge spaces to be circles, although this is not important for the subsequent argument.    We now construct a new graph of spaces $X'$ and an essential map $X'\to X$.

Let $k=\prod_u(\nu(u)-1)$, where the product is taken over all rigid vertices of $X$.  For each non-rigid vertex $X_v$, we take $k$ copies of the induced pair at $v$.  Consider a  rigid vertex $u$, with incident edges $\{e_i\}$.  For each edge $e_{i'}$ incident at $u$, we take $k/(\nu(u)-1)$ copies of  the pair $(X_u,\{X_{e_i}\mid i\neq i'\})$.  Note that every edge space of $X$ appears exactly $k$ times in this collection of pairs. We may therefore glue the resulting collection of pairs up to form a graph of spaces $X'$.

By construction, $X'$ is naturally equipped with a map $X'\to X$, which is $\pi_1$-injective by Proposition \ref{prop: Essential maps of graphs of groups}.  Furthermore, $X'$ is naturally equipped with exactly $k$ copies of component of $\w$; we call this map of circles $ \w_0$.   Let $\Gamma_0$ be the disjoint unions of the cores (in the sense of \cite{stallings_topology_1983}) of the covers of $\Gamma$ corresponding to the components of $X_0$.   We may realize $\w_0$ as a collection of cycles in $\Gamma_0$.  Then $(\Gamma_0,\w_0)\to(\Gamma,\w)$ is an admissible, essential map.   The pair $(\Gamma_0,\w_0)$ is irreducible by Lemma \ref{lem: Relative Shenitzer}, and hence by Lemma \ref{lem: Converse to Whitehead}, can be unfolded to a locally irreducible pair $(\Gamma',\w')\to(\Gamma,\w)$.

Finally, we compute Euler characteristics.  We have $n(\w')=k$, while
\[
\frac{\chi_-(\Gamma')}{k}=\sum_{u}\frac{\nu(u)}{(\nu(u)-1)}\chi_-(X_u) +\sum_v \chi_-(X_v)>\chi_-(\Gamma)~,
\]
(where $u$ ranges over all the rigid vertices of $X$ and $v$ ranges over all the non-rigid vertices).  This completes the proof.
\end{proof}
  
A similar argument shows that there are always maximal pairs of surface type.

\begin{lemma}\label{lem: Maximal surface type}
Let $(\Gamma,\w)$ be an irreducible graph pair.  If the relative JSJ of $(\Gamma,\w)$ has no rigid vertices then there is a locally irreducible group pair $(\Gamma',\w')$ of surface type and an admissible, essential $(\Gamma',\w')\to(\Gamma,\w)$ with
\[
\frac{\chi_-(\Gamma')}{n(\w')}=\chi_-(\Gamma)~.
\]
\end{lemma}
\begin{proof}
Consider the JSJ decomposition $\curlyG$ of the pair $(\Gamma,\w)$.  By Marshall Hall's theorem, we may assume that the attaching maps at cyclic vertices are all isomorphisms.  

We realize $\mathcal{G}$ as a graph of spaces $X$ in the natural way, taking each surface vertex to be a compact surface and each cyclic vertex to be a circle, and define a new graph of spaces $X'$ as follows.    We take $2$ copies of each surface vertex $X_v$. We take $\nu(w)$ copies of each non-peripheral cyclic vertex space $X_w$.  We take $2$ copies of each peripheral cyclic vertex space $X_w$ of $X$; these will each be a peripheral vertex of $X'$.   Finally, we take $\nu(w)-1$ further copies of each peripheral vertex space $X_w$; these will be non-peripheral vertices of $X'$.  It is now easy to see that we can assemble these to form $X'$ so that every non-peripheral cyclic vertex group has exactly two incident edges and all the attaching maps are isomorphisms.  As before, the natural map $X'\to X$ is $\pi_1$-injective by Proposition \ref{prop: Essential maps of graphs of groups}.

Every non-peripheral cyclic vertex $w$ is adjacent to exactly two surface vertices, and identified with two boundary components of these.  We may therefore contract the two edges adjacent to $w$ to obtain a larger surface vertex.  

Thus, the resulting graph of spaces $X'$ is homeomorphic to a surface $\Sigma$. The peripheral cyclic vertices equip $X'$ with exactly two copies of each component of $\w$; we call this collection of cyclic subgroups $\w'$.   We note that $n(\w')=2$ and that, since only surface vertices contribute to Euler characteristic, $\chi(\Sigma)=2\chi(F)$.  Therefore, if we replace 
$(\Sigma,\w')$ by a locally irreducible graph pair $(\Gamma',\w')$ as in the previous lemma, 
$\chi_-(\Gamma')/n(\w')=\chi_-(\Gamma)$.

At this stage, we have a locally irreducible, admissible pair $(\Gamma',\w')$ of surface type, satisfying the required constraints on Euler characteristic, so that each component of $\w'$ is conjugate into some component of $\partial\Sigma$ (and every component of $\partial\Sigma$ contains a component of $\w'$).  To make this pair of surface type, we need $\w'$ to be identified bijectively with $\partial\Sigma$.  To ensure this, we first invoke Marshall Hall's theorem again, replacing $\Gamma'$ with a finite-sheeted cover and $\w'$ with its pullback, so that each component of $\w'$ maps isomorphically to the component of $\partial\Sigma$ that contains it.  For each component $\partial_i\Sigma\subseteq\partial\Sigma$, let $n_i$ be the number of components of $\w'$ contained in $\partial_i\Sigma$.  Replacing $\Sigma$ with two copies of itself, we may assume that each $n_i$ is even. Without loss of generality, we may also assume that $n_1$ is minimal among the $n_i$. We now take $n_1$ copies of $\Sigma$, and equip each boundary component with exactly one component of $\w'$. We may then add annuli to $\Sigma$ to pair up the remaining components of $\w'$.  This completes the proof.
\end{proof}

We can now apply these two lemmas to prove that we can always find a maximal $\partial$-immersion of surface type.

\begin{proof}[Proof of Theorem \ref{thm: Maximal pairs and surfaces}]
 
By Lemma \ref{lem: Locally irreducible d-immersions}, $(\Delta,\vv)$ is weakly irreducible, so can be unfolded to a pair (without loss of generality, $(\Lambda,\uu)$) which is wedge of locally irreducible graph pairs $(\Lambda_j,\uu_j)$. That is, there is a finite set $\Xi$ equipped with maps $\xi_\pm:\Xi\to \coprod_j V(\Lambda_j)$, so that
\[
 \Lambda=\coprod_j \Lambda_j/\sim
\]
where $\xi_+(x)\sim \xi_-(x)$ for all $x\in\Xi$, and $\uu=\coprod_j\uu_j$.  For each $x\in\Xi$, let us fix choices of lifts of the maps $\xi_\pm$ to maps $\eta_\pm:\Xi\to \ess_{\uu}$.

Suppose that, for some $j$ (without loss of generality, $j=1$), the relative JSJ decomposition of the group pair $(\pi_1\Lambda_1,\uu_1)$ has a rigid vertex.   Lemma \ref{lem: Rigid implies non-maximal} applied to $(\Lambda_1,\uu_1)$  yields an essential map from an irreducible pair $(\Lambda'_1,\uu'_1)\to(\Lambda_1,\uu_1)$ that satisfies $\chi_-(\Lambda'_1)>d\chi_-(\Lambda_1)$, where $d$ is the degree of the covering map $\ess_{\uu'_1}\to\ess_{\uu_1}$.   By Lemma \ref{lem: Converse to Whitehead}, after unfolding, we may take the pair $(\Lambda'_1,\uu'_1)$  to be locally irreducible.  Let $(\Lambda'_j,\uu'_j)$ consist of $d$ copies of $(\Lambda_j,\uu_j)$, for each $j\neq 1$.  Let $(\widehat{\Lambda},\hat{\uu})=\coprod_j(\Lambda'_j,\uu'_j)$.  Let $\Xi'=\Xi\times\{1,\ldots,d\}$, and choose a map $\eta'_\pm:\Xi'\to\ess_{\hat{\uu}}$ so that, for each $x\in\Xi$, $\eta'_\pm(x,\cdot)$ indexes the $d$ preimages of $\eta_\pm(x)$ in $\ess_{\hat{\uu}}$.  Let $\xi'_\pm=\hat{\uu}\circ\eta'_\pm$.

The maps $\xi'_\pm$ now define a wedge $\Lambda'=\widehat{\Lambda}/\sim$, where $\xi'_+(x')\sim \xi'_-(x')$ for all $x'\in\Xi'$,  which folds (preserving Euler characteristic) to an immersion $\Delta'\to\Delta$.  Let $\uu'$ be the composition of $\hat{\uu}$ with the natural quotient map $\widehat{\Lambda}\to\Lambda'$, and let $\vv'$ the composition of $\uu'$ with the quotient map $\Lambda'\to\Delta'$.  We therefore have an admissible $\partial$-immersion
\[
(\Lambda',\uu')\to(\Delta',\vv')\immerses(\Delta,\vv)
\]
such that $(\Lambda',\uu')$ is locally irreducible with $\chi(\Delta')=\chi(\Lambda')$ and $n(\vv')=n(\uu')=dn(\vv)$. 

We now compare Euler characteristics:
\[
\chi_-(\Delta)=\sum_j\chi_-(\Lambda_j)+|\Xi|
\]
whereas
\[
\chi_-(\Delta')=\chi_-(\Lambda'_1)+\sum_{j\neq 1}\chi_-(\Lambda'_j)+|\Xi'|>d\chi_-(\Lambda_1)+d\sum_{j\neq 1}\chi_-(\Lambda'_j)+d|\Xi|~,
\]
so $\chi_-(\Delta')>d\chi_-(\Delta)$.    The $\partial$-immersion
\[
(\Lambda',\uu')\to (\Delta',\vv')\immerses(\Gamma,\w)
\]
therefore has greater projective $\curlyP$-rank, contradicting the maximality hypothesis.

The second part of the theorem follows in the same way, using Lemma \ref{lem: Maximal surface type} instead of Lemma \ref{lem: Rigid implies non-maximal}.
\end{proof}

Our main technical theorem follows immediately.

\begin{theorem}\label{thm: Essential surfaces exist}
Let $(\Gamma,\w)$ be an irreducible graph pair. There exists a compact surface with boundary $\Sigma$ and an admissible, essential map $(\Sigma,\partial\Sigma)\to(\Gamma,\w)$.
\end{theorem}
\begin{proof}
By Lemma \ref{lem: Maximal pairs exist}, a maximal, admissible $\partial$-immersion
\[
(\Lambda,\uu)\to(\Delta,\vv)\immerses(\Gamma,\w)
\]
exists.  By Theorem \ref{thm: Maximal pairs and surfaces}, there is such a maximal,  admissible $\partial$-immersion so that  $(\Delta,\vv)$ is of weak surface type. After passing to the disjoint union of the free factors, we obtain a maximal, admissible $\partial$-immersion so that $(\Delta,\vv)$ is of surface type.  
\end{proof}

\section{Surface subgroups and hierarchies}\label{sec: Hierarchies}

In this section we deduce the claimed consequences of Theorem \ref{thm: Essential surfaces exist}.  We start with graphs of virtually free groups with virtually cyclic edge groups.  The deduction of the existence of surface subgroups from a result like Theorem \ref{thm: Essential surfaces exist} is well known (cf.\ \cite{calegari_surface_2008} or \cite{wilton_one-ended_2011}, for instance); we include an argument here for completeness.

\begin{theorem}\label{thm: Surface subgroups}
Let $\Gamma$ be the fundamental group of a graph of virtually free groups with virtually cyclic edge groups. If $\Gamma$ is hyperbolic and one-ended then $\Gamma$ contains a quasiconvex surface subgroup.
\end{theorem}
\begin{proof}
By \cite{wise_subgroup_2000}, $\Gamma$ is residually finite and so virtually torsion-free. We may therefore assume that $\Gamma$ is the fundamental group of a graph of free groups with cyclic edge groups.  We call the vertices $v$ of the underlying graph of $\Gamma$ \emph{non-cyclic}, and subdivide each edge, putting a \emph{cyclic} vertex in the middle with vertex-group $\Z$.

Consider the induced pair $(F_v,\w_v)$ for a non-cyclic vertex $v$.  Since $\Gamma$ is one-ended, $(F_v,\w_v)$ is irreducible.   By Theorem \ref{thm: Essential surfaces exist}, we  can replace each $(F_v,\w_v)$ by an admissible, essential map of a surface pair $(\pi_1\Sigma_v,\partial\Sigma_v)$. By gluing these to the adjacent cyclic vertices, we define a new graph of free groups with cyclic edge groups $\curlyH$, with every non-cyclic vertex of surface type.  Note that $\chi(\Sigma_v)\leq 0$ for all $v$.

The fundamental group $H$ of $\curlyH$ is equipped with a natural map $f:H\to \Gamma$, and by Proposition \ref{prop: Essential maps of graphs of groups},  $f$ is injective.   In particular, $H$ contains no Baumslag--Solitar subgroups, since $\Gamma$ is hyperbolic.   

The graph of groups $\curlyH$ is a graph of surfaces glued along their boundary components to circles, which we realize in the natural way as a graph of spaces $X$.  Note that $\chi(X)=\sum_v \chi(\Sigma_v)\leq 0$.  By \cite{wise_subgroup_2000}, after replacing $H$ by a subgroup of finite index, we may assume that the attaching maps are all homeomorphisms.  It is then easy to see that $X$ can be thickened to a 3-manifold $M$ with boundary.  Since closed 3-manifolds have zero Euler characterisitic,
\[
\chi(\partial M)=2\chi(M)=2\chi(X)\leq 0
\]
and so we may choose a component $\Sigma$ of $\partial M$ with $\chi(\Sigma)\leq 0$.  Inclusion induces a natural map $\pi_1\Sigma\to H$.  But $H$ is one-ended by Shenitzer's lemma, and Dehn's lemma then implies that the map $\pi_1\Sigma\to H$ (and hence the composition $\pi_1\Sigma\to \Gamma$) is injective.

Finally, $\Gamma$ is locally quasiconvex \cite[Theorem D]{bigdely_quasiconvexity_2013}, and hence the surface subgroup $\pi_1\Sigma$ is quasiconvex.  This completes the proof.
\end{proof}

A group $\Gamma$ is called \emph{rigid} if it does not split over a (possibly finite) virtually cyclic subgroup.  Given a group $\Gamma$, a \emph{virtually cyclic hierarchy} for $\Gamma$ is a set of subgroups of $\Gamma$ obtained by passing to the vertex groups of a splitting of $\Gamma$ over virtually cyclic edge groups, and then repeating this operation on those subgroups recursively.  If a finite virtually cyclic hierarchy exists, terminating in (possibly finite) rigid subgroups, then we shall say that $\Gamma$ \emph{has a finite hierarchy}.   Graphs of virtually free groups with virtually cyclic edge groups play a special role in the subgroup theory of groups that have finite hierarchies.

\begin{remark}
Let $G$ be a one-ended group with a finite hierarchy, and let $H$ be a one-ended subgroup in the hierarchy of $G$ with no one-ended subgroups below it.  Then either $H$ is rigid or $H$ is a graph of virtually free groups over virtually cyclic edge groups.
\end{remark}

In  \cite{sela_diophantine_2001}, Sela showed that limit groups have a finite hierarchy.  He also showed that a limit group without a $\mathbb{Z}^2$ subgroup is hyperbolic, and that non-abelian limit groups are never rigid. We thus obtain:

\begin{corollary}\label{cor: Surface subgroups of limit groups}
Every one-ended limit group contains a surface subgroup.
\end{corollary}

Louder and Touikan showed that a hyperbolic group without 2-torsion has a finite hierarchy \cite{louder_strong_2017}.  (The restriction on 2-torsion is technical, and conjecturally can be removed.)  We thus obtain the following contribution towards the complete resolution of Gromov's question.

\begin{corollary}\label{cor: Surface or rigid}
Every one-ended hyperbolic group without 2-torsion either contains a surface subgroup or contains a quasiconvex, infinite, rigid subgroup.
\end{corollary}

In particular, Gromov's question is reduced to the rigid case (modulo the technical issue of 2-torsion).

\section{Applications to profinite rigidity}\label{sec: Profinite rigidity}

In this section we discuss applications to Question \ref{qu: Remeslennikov} and related problems.  As explained in \cite[Theorem 4.17]{bridson_determining_2016}, Corollary \ref{cor: Surface subgroups of limit groups} resolves the question for limit groups.

We include the proof for completeness.  The key point is that it follows that the profinite completions of non-free limit groups have non-zero virtual second cohomology.  We work with continuous cohomology, with coefficients in $\Z/2$.

\begin{theorem}\label{thm: Profinite vb2 for limit groups}
If $L$ is a limit group and not free then there is a subgroup $L_0$ of finite index in $L$ such that $H^2(\wh{L}_0;\Z/2)\neq 0$.
\end{theorem}
\begin{proof}
By Corollary \ref{cor: Surface subgroups of limit groups}, $L$ contains a subgroup $S$ isomorphic to the fundamental group of a closed surface of non-positive Euler characteristic; in particular, $H^2(S)\neq 0$ (with coefficients in $\Z/2$).   Since surface groups are \emph{good} in the sense of Serre \cite{grunewald_cohomological_2008}, it follows that the continuous cohomology $H^2(\wh{S})$ is also non-zero.  By \cite{wilton_halls_2008}, $S$ is a virtual retract of $L$, so there is a finite-index subgroup $L_0$ containing $S$ and a retraction $r:L_0\to S$.  Let $i:S\to L_0$ be the inclusion map, so $r\circ i=\id_S$. Both $r$ and $i$ extend by continuity to maps $\hat{r}:\wh{L}_0\to\wh{S}$ and $\hat{i}:\wh{S}\to\wh{L}_0$, and $\hat{r}\circ\hat{i}=\id_{\wh{S}}$.  Therefore the induced maps on cohomology satisfy $\hat{i}^*\circ\hat{r}^*=\id_{H^*(\wh{S})}$.  In particular, since $H^2(\wh{S})$ is non-zero, $H^2(\wh{L}_0)$ is also non-zero, as claimed.
\end{proof}

Since every open subgroup of a profinite free group is profinite free, and hence has zero second cohomology, Corollary \ref{introcor: Remeslennikov for limit groups} follows immediately.  Corollary \ref{introcor: PP} also follows quickly from Theorem \ref{thm: Profinite vb2 for limit groups}

\begin{proof}[Proof of Corollary \ref{introcor: PP}]
To prove the contrapositive, we assume that $\w$ is not primitive in $F$.   Recall that the \emph{double} $D(\w)$ is the fundamental group of the graph of groups with two vertices labelled by $F$, and one edge between them for each component of $\w$.  The double $D(\w)$ is a limit group and hence, by Theorem \ref{thm: Profinite vb2 for limit groups}, has a subgroup $D_0$ of finite index with $H^2(\wh{D}_0;\Z/2)\neq 0$.  If $\w$ were primitive in $\wh{F}$ then $\wh{D(\w)}$ would be free profinite, hence so would $\wh{D}_0$, and therefore $H^2(\wh{D}_0;\Z/2)$ would be zero, a contradiction.
\end{proof}

\bibliographystyle{plain}

\Addresses

\end{document}

%% file: bsword.pdf_tex
%% Creator: Inkscape inkscape 0.91, www.inkscape.org
%% PDF/EPS/PS + LaTeX output extension by Johan Engelen, 2010
%% Accompanies image file 'bsword.pdf' (pdf, eps, ps)
%%
%% To include the image in your LaTeX document, write
%%   \input{<filename>.pdf_tex}
%%  instead of
%%   \includegraphics{<filename>.pdf}
%% To scale the image, write
%%   \def\svgwidth{<desired width>}
%%   \input{<filename>.pdf_tex}
%%  instead of
%%   \includegraphics[width=<desired width>]{<filename>.pdf}
%%
%% Images with a different path to the parent latex file can
%% be accessed with the `import' package (which may need to be
%% installed) using
%%   \usepackage{import}
%% in the preamble, and then including the image with
%%   \import{<path to file>}{<filename>.pdf_tex}
%% Alternatively, one can specify
%%   \graphicspath{{<path to file>/}}
%% 
%% For more information, please see info/svg-inkscape on CTAN:
%%   http://tug.ctan.org/tex-archive/info/svg-inkscape
%%
\begingroup%
  \makeatletter%
  \providecommand\color[2][]{%
    \errmessage{(Inkscape) Color is used for the text in Inkscape, but the package 'color.sty' is not loaded}%
    \renewcommand\color[2][]{}%
  }%
  \providecommand\transparent[1]{%
    \errmessage{(Inkscape) Transparency is used (non-zero) for the text in Inkscape, but the package 'transparent.sty' is not loaded}%
    \renewcommand\transparent[1]{}%
  }%
  \providecommand\rotatebox[2]{#2}%
  \ifx\svgwidth\undefined%
    \setlength{\unitlength}{301.53348667bp}%
    \ifx\svgscale\undefined%
      \relax%
    \else%
      \setlength{\unitlength}{\unitlength * \real{\svgscale}}%
    \fi%
  \else%
    \setlength{\unitlength}{\svgwidth}%
  \fi%
  \global\let\svgwidth\undefined%
  \global\let\svgscale\undefined%
  \makeatother%
  \begin{picture}(1,0.54450062)%
    \put(0,0){\includegraphics[width=\unitlength,page=1]{bsword.pdf}}%
    \put(-0.00880909,0.39048859){\color[rgb]{0,0,0}\makebox(0,0)[lb]{\smash{$b$}}}%
    \put(0.7338728,0.39951309){\color[rgb]{0,0,0}\makebox(0,0)[lb]{\smash{$a$}}}%
    \put(0,0){\includegraphics[width=\unitlength,page=2]{bsword.pdf}}%
    \put(0.45387055,0.03928576){\color[rgb]{0,0,0}\makebox(0,0)[lb]{\smash{$a$}}}%
    \put(0.46145086,0.15867548){\color[rgb]{0,0,0}\makebox(0,0)[lb]{\smash{$\bar{a}$}}}%
    \put(0.22361896,0.02507276){\color[rgb]{0,0,0}\makebox(0,0)[lb]{\smash{$b$}}}%
    \put(0.22361896,0.15583283){\color[rgb]{0,0,0}\makebox(0,0)[lb]{\smash{$\bar{b}$}}}%
  \end{picture}%
\endgroup%

%% file: fold.pdf_tex
%% Creator: Inkscape inkscape 0.91, www.inkscape.org
%% PDF/EPS/PS + LaTeX output extension by Johan Engelen, 2010
%% Accompanies image file 'fold.pdf' (pdf, eps, ps)
%%
%% To include the image in your LaTeX document, write
%%   \input{<filename>.pdf_tex}
%%  instead of
%%   \includegraphics{<filename>.pdf}
%% To scale the image, write
%%   \def\svgwidth{<desired width>}
%%   \input{<filename>.pdf_tex}
%%  instead of
%%   \includegraphics[width=<desired width>]{<filename>.pdf}
%%
%% Images with a different path to the parent latex file can
%% be accessed with the `import' package (which may need to be
%% installed) using
%%   \usepackage{import}
%% in the preamble, and then including the image with
%%   \import{<path to file>}{<filename>.pdf_tex}
%% Alternatively, one can specify
%%   \graphicspath{{<path to file>/}}
%% 
%% For more information, please see info/svg-inkscape on CTAN:
%%   http://tug.ctan.org/tex-archive/info/svg-inkscape
%%
\begingroup%
  \makeatletter%
  \providecommand\color[2][]{%
    \errmessage{(Inkscape) Color is used for the text in Inkscape, but the package 'color.sty' is not loaded}%
    \renewcommand\color[2][]{}%
  }%
  \providecommand\transparent[1]{%
    \errmessage{(Inkscape) Transparency is used (non-zero) for the text in Inkscape, but the package 'transparent.sty' is not loaded}%
    \renewcommand\transparent[1]{}%
  }%
  \providecommand\rotatebox[2]{#2}%
  \ifx\svgwidth\undefined%
    \setlength{\unitlength}{748.08884027bp}%
    \ifx\svgscale\undefined%
      \relax%
    \else%
      \setlength{\unitlength}{\unitlength * \real{\svgscale}}%
    \fi%
  \else%
    \setlength{\unitlength}{\svgwidth}%
  \fi%
  \global\let\svgwidth\undefined%
  \global\let\svgscale\undefined%
  \makeatother%
  \begin{picture}(1,0.7303402)%
    \put(0,0){\includegraphics[width=\unitlength,page=1]{fold.pdf}}%
    \put(0.09506188,0.59197634){\color[rgb]{0,0,0}\makebox(0,0)[lb]{\smash{$x$}}}%
    \put(0.26142018,0.69784072){\color[rgb]{0,0,0}\makebox(0,0)[lb]{\smash{$y_1$}}}%
    \put(0.2787042,0.48395153){\color[rgb]{0,0,0}\makebox(0,0)[lb]{\smash{$y_2$}}}%
    \put(0.74321114,0.59413685){\color[rgb]{0,0,0}\makebox(0,0)[lb]{\smash{$x'$}}}%
    \put(0.87284094,0.59629734){\color[rgb]{0,0,0}\makebox(0,0)[lb]{\smash{$y'$}}}%
    \put(0.07345691,0.22685232){\color[rgb]{0,0,0}\makebox(0,0)[lb]{\smash{$\Wh_x(\w)$}}}%
    \put(0.35000059,0.25925977){\color[rgb]{0,0,0}\makebox(0,0)[lb]{\smash{$\Wh_{y_1}(\w)$}}}%
    \put(0.35432156,0.03672854){\color[rgb]{0,0,0}\makebox(0,0)[lb]{\smash{$\Wh_{y_2}(\w)$}}}%
    \put(0.7043222,0.21388934){\color[rgb]{0,0,0}\makebox(0,0)[lb]{\smash{$\Wh_{x'}(\w')$}}}%
    \put(0.89660638,0.20956831){\color[rgb]{0,0,0}\makebox(0,0)[lb]{\smash{$\Wh_{y'}(\w')$}}}%
    \put(0,0){\includegraphics[width=\unitlength,page=2]{fold.pdf}}%
  \end{picture}%
\endgroup%

%% file: rigid.pdf_tex
%% Creator: Inkscape inkscape 0.91, www.inkscape.org
%% PDF/EPS/PS + LaTeX output extension by Johan Engelen, 2010
%% Accompanies image file 'rigid.pdf' (pdf, eps, ps)
%%
%% To include the image in your LaTeX document, write
%%   \input{<filename>.pdf_tex}
%%  instead of
%%   \includegraphics{<filename>.pdf}
%% To scale the image, write
%%   \def\svgwidth{<desired width>}
%%   \input{<filename>.pdf_tex}
%%  instead of
%%   \includegraphics[width=<desired width>]{<filename>.pdf}
%%
%% Images with a different path to the parent latex file can
%% be accessed with the `import' package (which may need to be
%% installed) using
%%   \usepackage{import}
%% in the preamble, and then including the image with
%%   \import{<path to file>}{<filename>.pdf_tex}
%% Alternatively, one can specify
%%   \graphicspath{{<path to file>/}}
%% 
%% For more information, please see info/svg-inkscape on CTAN:
%%   http://tug.ctan.org/tex-archive/info/svg-inkscape
%%
\begingroup%
  \makeatletter%
  \providecommand\color[2][]{%
    \errmessage{(Inkscape) Color is used for the text in Inkscape, but the package 'color.sty' is not loaded}%
    \renewcommand\color[2][]{}%
  }%
  \providecommand\transparent[1]{%
    \errmessage{(Inkscape) Transparency is used (non-zero) for the text in Inkscape, but the package 'transparent.sty' is not loaded}%
    \renewcommand\transparent[1]{}%
  }%
  \providecommand\rotatebox[2]{#2}%
  \ifx\svgwidth\undefined%
    \setlength{\unitlength}{1888.78250723bp}%
    \ifx\svgscale\undefined%
      \relax%
    \else%
      \setlength{\unitlength}{\unitlength * \real{\svgscale}}%
    \fi%
  \else%
    \setlength{\unitlength}{\svgwidth}%
  \fi%
  \global\let\svgwidth\undefined%
  \global\let\svgscale\undefined%
  \makeatother%
  \begin{picture}(1,0.23846711)%
    \put(0,0){\includegraphics[width=\unitlength,page=1]{rigid.pdf}}%
    \put(-0.00140633,0.15802287){\color[rgb]{0,0,0}\makebox(0,0)[lb]{\smash{$X'$}}}%
    \put(0,0){\includegraphics[width=\unitlength,page=2]{rigid.pdf}}%
    \put(0.67631338,0.14604297){\color[rgb]{0,0,0}\makebox(0,0)[lb]{\smash{$X$}}}%
    \put(0.74819274,0.00399562){\color[rgb]{0,0,0}\makebox(0,0)[lb]{\smash{$w_1$}}}%
    \put(0.881683,0.005707){\color[rgb]{0,0,0}\makebox(0,0)[lb]{\smash{$w_2$}}}%
    \put(0.94671672,0.00741848){\color[rgb]{0,0,0}\makebox(0,0)[lb]{\smash{$w_3$}}}%
    \put(0.74819274,0.00399562){\color[rgb]{0,0,0}\makebox(0,0)[lb]{\smash{$w_1$}}}%
    \put(0.881683,0.005707){\color[rgb]{0,0,0}\makebox(0,0)[lb]{\smash{$w_2$}}}%
    \put(0.94671672,0.00741848){\color[rgb]{0,0,0}\makebox(0,0)[lb]{\smash{$w_3$}}}%
    \put(0.07389585,0.01255272){\color[rgb]{0,0,0}\makebox(0,0)[lb]{\smash{$w_1$}}}%
    \put(0.36825895,0.01255267){\color[rgb]{0,0,0}\makebox(0,0)[lb]{\smash{$w_2$}}}%
    \put(0.43329267,0.01426415){\color[rgb]{0,0,0}\makebox(0,0)[lb]{\smash{$w_3$}}}%
  \end{picture}%
\endgroup%